\documentclass{article}
\usepackage[utf8]{inputenc}
\usepackage{amsfonts}
\usepackage{appendix}
\usepackage{amsthm}
\usepackage{mathrsfs}

\usepackage{yfonts}

\usepackage{amsmath}
\usepackage{dsfont}
\usepackage{amssymb}
\usepackage{graphicx}
 \usepackage{color, xcolor}
\usepackage{array}
\usepackage{fullpage}
\usepackage{caption}
\usepackage{subcaption}
\newcolumntype{R}[1]{>{\raggedleft\arraybackslash }b{#1}}
\newcolumntype{L}[1]{>{\raggedright\arraybackslash }b{#1}}
\newcolumntype{C}[1]{>{\centering\arraybackslash }b{#1}}

\newtheorem{thm}{Theorem}[section]
\newtheorem{Def}[thm]{Definition}
\newtheorem{prop}[thm]{Proposition}
\newtheorem{lem}[thm]{Lemma}
\newtheorem{cor}[thm]{Corollary}
\newtheorem{rem}[thm]{Remark}

\newcommand{\R}{\mathbb{R}}

\newcommand{\N}{\mathbb{N}}

\newcommand{\E}{\mathbb{E}}
\newcommand{\inte}{\int_0^t}

\newcommand{\e}{\eta}
\newcommand{\LL}{\mathcal{L}}

\newcommand{\lims}{\underset{N\rightarrow \infty}{\overline{\lim}}}
\newcommand{\nurhon}{\nu_{\rho}^N}

\newcommand{\sumn}{\sum_{x=1}^{N-1}}

\numberwithin{equation}{section}

\usepackage{geometry} 

\usepackage{soul}
\setstcolor{red}

\title{\textbf{Dynamic and static large deviations of a one
dimensional SSEP in weak contact with reservoirs }} \date{}

\usepackage{authblk}

\newsavebox\affbox
\author[1]{\text{Claudio Landim}}
\author[2]{\text{Sonia Velasco}}
\affil[1] {IMPA, \textit{Estrada Dona Castorina 110, J. Botanico, 22460 Rio de Janeiro, Brazil.}
CNRS UMR 6085, \textit{Université de Rouen, 76128 Mont-Saint-Aignan Cedex, France.}}
\affil[2]{ Laboratoire MAP5, \textit{ Université Paris Cité, 45 rue des Saints-Pères, 75270 Paris Cedex 06, France.}
}

\begin{document}

\maketitle

\begin{abstract} 
We derive a formula for the quasi-potential of one-dimensional
symmetric exclusion process in weak contact with reservoirs. The
interaction with the boundary is so weak that, in the diffusive scale,
the density profile evolves as the one of the exclusion process with
reflecting boundary conditions. In order to observe an evolution of
the total mass, the process has to be observed in a longer time-scale,
in which the density profile becomes immediately constant. 
\end{abstract}

\let\thefootnote\relax\footnotetext{\noindent \small $^{*}$\textbf{
Corresponding author: sonia.velasco@parisdescartes.fr}}



\section{Introduction}

The investigation of the nonequilibrium stationary states of
interacting particle systems in contact with reservoirs has attracted
a lot of attention in the last years \cite{BDGJL, DLS2002, BDGJL2015,
DHS2021, BEL21, cfggt}.

In this article, we examine the one-dimensional symmetric exclusion
process in weak contact with reservoirs. The interaction of the system
with the reservoirs is so weak that the hydrodynamic equation describing the macroscopic time evolution of the empirical density is
the heat equation with Neumann boundary conditions, the equation which
represents the density evolution of the exclusion process with
reflecting boundary conditions.

The total mass changes in a much longer macroscopic time-scale. In
this very long time-scale, the empirical density immediately reaches 
the stationary profile of the heat equation with Neumann boundary
conditions, that is, a constant profile with a certain time-dependent
value.

There are therefore two important time-scales. In the first one, the
density profile evolves according to the heat equation with Neumann
boundary conditions and converges, as time increases,
to a constant density profile without modifying its total mass. In the second much longer time-scale, the constant density profile evolves smoothly, modifying the total mass according to an ODE, until it reaches the
value determined by the interactions of the system with the boundary.

This picture extends to the dynamical large deviations.  Denote by
$K_{[0,T]} (u)$ the cost of observing a trajectory $u(t)$ in the time
interval $[0,T]$. For each $t$, $u(t)$ is a density profile.  Assume
that the total mass of $u(t)$ is constant in time
($\int_K u(t,x) dx = c$). As the interaction with the boundary is
small, this cost coincides with that of an exclusion process with
reflection at the boundary \cite{GDnoncritical}. Denote the later cost
by $K^{\rm Neu}_{[0,T]} (u)$, so that
$K_{[0,T]} (u) = K^{\rm Neu}_{[0,T]} (u)$.

To consider the large deviations of the total mass, observe the system
in the longer time-scale at which the total mass evolves. Fix a
trajectory $a\colon [0,T'] \to [0, M]$, where $M$ represents the
length (or volume) of the interval where particles are interacting,
and denote by $I_{[0,T']} (a)$ the cost of observing a trajectory
$u(t)$, $0\le t\le T'$, whose total mass at time $t$ is equal to
$a(t)$. Theorem \ref{DynamicGD} below states the dynamical large
deviations principle for the total mass and provides a formula for
$I _{[0,T']} (a)$.

We turn to the quasi-potential. The existence of two time-scales creates
an obstruction in its derivation.  Denote by $\bar\rho_m$,
$0\le m\le M$, the constant density profile with total mass equal to
$m$.  Let $V^{\rm Neu}_m (\cdot)$ be the quasi-potential associated to
the exclusion dynamics with reflection at the boundary: for a density
profile $\rho$ whose total mass is equal to $m$,
$V^{\rm Neu}_m (\rho) = \inf_{T>0} \inf_u K^{\rm Neu}_{[0,T]} (u)$,
where the second infimum is carried over all trajectories $u$ with
time-independent total mass and such that $u(0) = \bar\rho_m$,
$u(T) = \rho$.

Fix a density profile $\rho$ whose total mass is equal to $m$, and
let $u^{\rm Neu}_\rho$ be the relaxation trajectory, that is, the
trajectory which describes the typical evolution of the density
profile when the initial condition is $\rho$. As this evolution
corresponds to the solution of the heat equation with Neumann boundary
conditions, $u^{\rm Neu}_\rho (t) \to \bar\rho_m$ as
$t\to\infty$. Let $u^{\rm R, Neu}_\rho$ be the time-reflected
trajectory: $u^{\rm R, Neu}_\rho (t) = u^{\rm Neu}_\rho (-t) $.
As the exclusion process with reflecting boundary conditions is
reversible,
$V^{\rm Neu}_m (\rho) = K^{\rm Neu}_{(-\infty, 0]} (u^{\rm R,
Neu}_\rho)$.

Since the dynamical large deviations rate functional of the exclusion
process with weak interaction at the boundary coincides with the one
with reflection at the boundary, the quasi-potentials also coincide. The infimum is therefore reached at the time-reversed relaxation
trajectory: for any density profile $\rho$ with density $m$,
\begin{equation*}
V_m (\rho) \;:=\;
\inf_{T>0} \inf_u K_{[0,T]} (u) \;=\;
\inf_{T>0} \inf_u K^{\rm Neu}_{[0,T]} (u) \;=\;
K^{\rm Neu}_{(-\infty, 0]} (u^{\rm R, Neu}_\rho)
\;=\; V^{\rm Neu}_m (\rho)\;.
\end{equation*}

Moreover, as the stationary states of the exclusion process with
reflecting boundary conditions are the uniform measures with fixed
total number of particles,
\begin{equation*}
V_m (\rho) \;:=\;
V^{\rm Neu}_m (\rho) \;=\; \int_K \Big\{\, \rho (x)
\,\log \frac{\rho(x)}{m} \,+\,
[1-\rho (x)]
\,\log \frac{1-\rho(x)}{1-m} \,\Big\}\, dx 
\end{equation*}
for any density profile $\rho(\cdot)$ whose total mass is equal to
$m$.

The quasi-potential $V(\cdot)$ associated to the mass evolution, is given by $V(m) := \inf_{T>0} \inf_a I_T (a)$, where the
second infimum is carried over all trajectories
$a\colon [0,T] \to [0,M]$ such that $a(0) = \gamma$, $a(T)= m$.
In this formula, $\gamma$ stands for the typical mass under the
stationary state, determined by the weak interaction of the system
with the boundary.

Fix a mass $m$ and let $a_m$ be the relaxation path of the total mass
which starts from $m$ and converges to $\gamma$ as $t\to\infty$. Denote
by  $a^{\rm R}_m$ the time-reversed trajectory. In Lemma
\ref{Identific} and its proof we show that
\begin{equation*}
V(m) \;=\; I_{(-\infty, 0]} (a^{\rm R}_m) \;=\;
M \, \Big\{\, m
\,\log \frac{m}{\gamma} \,+\,
[1-m]
\,\log \frac{1-m}{1-\gamma} \,\Big\}\;. 
\end{equation*}
Theorem \ref{StaticGD} asserts that $V$ is indeed the large deviations
rate function of the total mass under the stationary state.

We conclude this section with heuristics to derive the quasi-potential
for the symmetric exclusion process with weak interaction at the
boundary.  Fix a density profile $\rho(\cdot)$ with mass $m$. As
time evolves, it relaxes to a constant density equal to
$m$. In a longer time-scale the total mass relaxes to $\gamma$.
Hence, reasoning backwards, it is expected that to fluctuate to $\rho$,
the system first changes its mass from $\gamma$ to $m$ following the relaxation path reflected in time. Then, its density profile evolves from one which is constant in space and has total mass
equal to $m$, to $\rho$, following the relaxation path reflected in
time. The total cost of this trajectory is given by
\begin{equation*}
I_{(-\infty, 0]} (a^{\rm R}_m) \;+\;
K^{\rm Neu}_{(-\infty, 0]} (u^{\rm R, Neu}_\rho) \;=\;
V(m) \,+\, V_m (\rho)   \;.
\end{equation*}
By the previous identities,
\begin{equation*}
W(\rho) \;:=\; V(m) \,+\, V_m (\rho)  \;=\;
\int_K \Big\{\, \rho (x)
\,\log \frac{\rho(x)}{\gamma} \,+\,
[1-\rho (x)]
\,\log \frac{1-\rho(x)}{1-\gamma} \,\Big\}\, dx \;,
\end{equation*}
is the quasi-potential for the symmetric exclusion process with weak
interaction at the boundary.

To prove that $W$ is indeed the quasi-potential, one should consider a
time-inhomogeneous dynamics in which the process evolves diffusively
in a time interval $[0,T]$, and properly time-rescaled, so as to observe an
evolution of the total mass in a time interval $[T,T+S]$.

It is also possible to use the matrix representation of the stationary
state to derive the above equation for the quasi-potential \cite{D22}. It
coincides with the quasi-potential for the symmetric exclusion process
with Robin boundary conditions as the interaction with the boundary
vanishes \cite{DHS2021, BEL21, BL22}.

\section{Notation and results}
\label{sec2}

Denote the state space by $\Omega_N:=\{0,1\}^{\Sigma_N}$, $N\ge 1$,
where $\Sigma_N=\{1,...,N-1\}$, and by $\e$ its element, so that for
any $x\in \Sigma_N$, $\e(x)=1$ if site $x$ is occupied and $\e(x)=0$
if it is empty.

Consider the infinitesimal generator
$\mathcal{L}_N=N^2\mathcal{L}_{N,0} + \mathcal{L}_{N,b}$ defined as
follows. For any function $f:\Omega_N \rightarrow \R$,
$$
(\mathcal{L}_{N,0}f)(\eta) = \sum_{x=1}^{N-2}\left(f(\eta^{x,x+1})-f(\eta)\right), $$
$$
(\mathcal{L}_{N,b}f)(\eta)= \sum_{x\in
\{1,N-1\}}\left[r_x(1-\eta(x))+(1-r_x)\eta(x)
\right]\left(f(\eta^x)-f(\eta) \right)
$$
with $r_1=\alpha$, $r_{N-1}=\beta$, where $0<\alpha,\beta<1$. In this
formula, for $x\in \{1,...,N-2\}$, the configuration $\eta^{x,x+1} $
is obtained from $\eta$ by exchanging the occupation variables
$\eta(x)$ and $\eta(x+1)$, i.e.,
\begin{equation} 
   \eta^{x,x+1}(y)= \left\{
    \begin{array}{ll}
        \e(x+1),~ ~\text{if}~ y=x\\ 
        \e(x),~ ~\text{if}~ y=x+1\\
        \e(y),~ ~\text{otherwise},
        \end{array}
\right.
\end{equation}
and for $x\in \{1,N-1\}$, the configuration $\e^x$ is obtained from $\e$ by flipping the occupation variable $\e(x)$, i.e,
\begin{equation} 
   \e^x(y)= \left\{
    \begin{array}{ll}
        1-\e(y),~ ~\text{if}~ y=x\\ 
        \e(y),~ ~\text{otherwise}.
        \end{array}
\right.
\end{equation}

\subsection*{Hydrodynamical limit and large deviations}

Here, we recall the results established in \cite{SSEPslowboundary},
resp. \cite{GDnoncritical} regarding the hydrodynamic limit,
resp. large deviations for the process with generator $\mathcal{L}_N$
known as the SSEP with slow boundaries. For that, let us first
introduce some notation. We fix a time horizon $T>0$.

Given a metric space $A$, $\mathcal{D}_A^T$ denotes the space of trajectories on $[0,T]$ which are right continuous with left limits and with values in $A$. Given a measure $\mu_N$ on $\Omega_N$, $\mathbb{P}_{\mu_N}$ is
the probability measure induced on $\mathcal{D}_{\Omega_N}^T$ by
$\{\eta_t,~ t\in [0,T]\}$ when  $\e_0$ has law $\mu_N$. Also,
denote by $\mathbb{E}_{\mu_N}$ the expectation with respect to
$\mathbb{P}_{\mu_N}$. Introduce
$$
\mathcal{M}:= \{\mu~ \text{is a positive measure on }~
[0,1]~\text{such that} ~ \mu([0,1])\leq 1\},
$$ which we equip with the weak topology. Denote by $\mathcal{M}_0\subset\mathcal{M}$ the subset of elements
which are absolutely continuous with respect to the Lebesgue measure
and with density between zero and one.

Introduce the empirical measure associated to an element $\e \in \Omega_N$ as the the element of $ \mathcal{M}$ defined by
\begin{equation*}
    \pi^N(\e,\mathrm{d}{u}) =\pi^N(\mathrm{d}{u}) := \frac{1}{N}\sum_{x=1}^{N-1}\e(x)\delta_{\frac{x}{N}}(\mathrm{d}{u})
\end{equation*}
where $ \delta_{\frac{x}{N}}$ is the Dirac measure at point $x/N$.  The process $(\pi_t^N)_{0\leq t\leq T}$ of empirical measures associated to $(\e_t)_{0\leq t \leq T}$ is a Markov process inducing a measure on $\mathcal{D}_{\mathcal{M}}^T$.

For a continuous function
$F:[0,1]\rightarrow \R$, write
$$\langle\pi^N,F\rangle = \frac{1}{N}\sumn\e(x)F(x/N) $$
and for $F,G\in L^2([0,1])$,
$$\langle F,G\rangle = \int_0^1F(x)G(x)dx. $$

Denote by $C^{i,j} = C^{i,j}([0,T]\times [0,1])$ the set of functions
that are of class $C^{i}$ in time and of class $C^{j}$ in space with
$i,j$ possibly infinite, and $C^{i}([0,T])$ resp. $C^{i}([0,1])$ the
functions that are $C^{i}$ in time, resp. space. To lighten notation,
given $G : [0,T]\times[0,1]\rightarrow \R$, we will sometimes write
$G_t(x)$ instead of $G(t,x)$ and $\partial_t G(t,x)$ denotes its
partial time derivative.

We say that a sequence of probability measures $(\mu_N)_{N\geq 1}$ on
$\Omega_N$ is associated to the profile $\rho_0\colon [0,1]\to [0,1]$
if for any $\delta>0$ and any continuous function G,
\begin{equation*}
    \underset{N \rightarrow \infty}{\lim} \mu_N\left[ \left|\langle \pi^N,G \rangle -\langle \rho_0,G\rangle \right|>\delta \right]=0.
\end{equation*}

The following results established in \cite{SSEPslowboundary} and
\cite{GDnoncritical} hold:

\begin{thm}
\label{LGNnonacc}(Hydrodynamic limit, c.f. \cite{SSEPslowboundary})
Consider a sequence $(\mu_N)_{N\geq 1}$ associated to a measurable
profile $\rho_0\colon [0,1]\to [0,1]$. Then, for every $t\geq 0$,
$\delta>0$ and continuous function $f:[0,1] \rightarrow \R$,
\begin{equation*}
\underset{N \rightarrow \infty}{\lim}~
\mathbb{P}_{\mu_N}\left[~ \left|
\langle \pi_t^N,f\rangle - \langle \rho_t,f\rangle > \delta \right| ~ \right] = 0,
\end{equation*}
where $\rho(t,.)$ is the unique weak solution of the heat equation with Neumann boundary conditions
 \begin{equation} \label{heatNeumann}
   \left\{
    \begin{array}{ll}
        \partial_t\rho(t,x) = \partial^2_x \rho(t,x),~ ~ \text{for}~ t>0,~ x\in (0,1),\\ 
        \partial_x\rho(t,0) = \partial_x\rho(t,1)=0,~ ~ \text{for}~ t>0,\\
        \rho(0,x) = \rho_0(x),~ ~ \text{for}~ x\in [0,1].
        \end{array}
\right.
\end{equation}
\end{thm}

\indent We write $L^2(0,1)$ the space of $L^2$ functions defined on
$[0,1]$ and denote by $\|. \|_{L^2(0,1)}$, the associated norm. Let
$\mathcal{H}^1(0,1)$ be the subset of $L^2(0,1)$ such that for any
$F\in \mathcal{H}^1(0,1)$, there is $\partial_xF\in L^2(0,1)$
satisfying
$\langle \partial_xF,G \rangle = -\langle F,\partial_x G \rangle$, for
any $G\in C^{\infty}$ with compact support in $(0,1)$. For
$F\in \mathcal{H}^1(0,1)$ define the norm
$$\|F\|_{\mathcal{H}^1}:= \left(\|F\|_{L^2}^2 + \|\partial_x F\|_{L^2}^2 \right)^{1/2}. $$

Let $L^2([0,T],\mathcal{H}^1)$ be the space of measurable functions $F:[0,T] \rightarrow \mathcal{H}^1$ such that
$$\|F\|^2_{L^2([0,T],\mathcal{H}^1)}:= \int_0^T \|F_t\|^2_{\mathcal{H}^1}dt<\infty. $$
Define the energy functional $\mathcal{E}:\mathcal{D}_{\mathcal{M}}^T \rightarrow [0,\infty] $, as in \cite{GDnoncritical}, by
$\mathcal{E}(\pi) = \underset{H}{\sup}~ \mathcal{E}_H(\pi) $ with the supremum taken over elements in $C^{0,1}$ with compact support and where,
 \begin{equation*} 
   \mathcal{E}_H(\pi) = \left\{
    \begin{array}{ll}
        \int_0^1\int_0^T\partial_xH(t,x)\rho(t,x)dtdx - 2\int_0^1\int_0^TH^2(t,x),~ ~ \text{if}~ \pi \in \mathcal{D}_{\mathcal{M}_0}^T~ \text{and}~ \pi_t(dx) = \rho_t(x)dx\\\\
        + \infty~ ~ \text{otherwise}.
        \end{array}
\right.
\end{equation*}
Introduce
$$\mathcal{F} = \left\{ \pi \in \mathcal{D}_{\mathcal{M}}^T,~  \langle \pi_t,1\rangle =  \langle \pi_0,1\rangle,~ \forall t\in [0,T]\right\}. $$
For $H\in C^{1,2}$, define the linear functional $\Tilde{J}_H(\pi):\mathcal{D}_{\mathcal{M}}^T \rightarrow [0,\infty] $ as
 \begin{equation*} 
    \Tilde{J}_H(\pi) = \left\{
    \begin{array}{ll}
        \langle \rho_T,H_T\rangle-\langle \rho_0,H_0\rangle - \int_0^T\langle \rho_s,\partial_s H_s\rangle ds ~+ &\int_0^T\langle \partial_x \rho_s, \partial_x H_s\rangle ds-\int_0^T \langle \rho_s(1-\rho_s),\left(\partial_x H_s \right)^2\rangle ds\\\\
        &\text{if}~ ~ \pi \in \mathcal{F}~ \text{and}~ \mathcal{E}(\pi)<\infty~ \text{with}~ \pi_t= \rho_t(x)dx,\\\\
        + \infty~ ~ \text{otherwise},
        \end{array}
\right.
\end{equation*}
The rate function
$\Tilde{I}_T: \mathcal{D}_{\mathcal{M}}^T\to [0,+\infty]$ that appears
in the large deviations principle proved in \cite{GDnoncritical} and
recalled below (Theorem \ref{GDNoncritical}) is given by:
$$\Tilde{I}_T(\pi) = \underset{H\in C^{1,2}}{\sup} \Tilde{J}_H(\pi). $$

\begin{thm} (Large deviations principle, c.f. \cite{GDnoncritical}) \label{GDNoncritical} Consider a sequence of deterministic configurations $(\e^N)_{N\geq 1}$ associated to a continuous profile $\rho_0$ which is bounded away from $0$ and $1$. The sequence of probability measures $(\mathbb{P}_{\delta_{\e^N}})_{N\geq 1}$ satisfies the following large deviations principle:
\begin{itemize}
    \item [(i)]\textbf{(Upper bound)} For any closed subset $\mathcal{C}$ of $\mathcal{D}_{\mathcal{M}}^T$,
    \begin{equation*}
        \lims ~ \frac{1}{N}\log \mathbb{P}_{\delta_{\e^N}}\left[\mathcal{C} \right] \leq - \underset{\pi \in \mathcal{C}}{\inf}~ \Tilde{I}_T(\pi)
    \end{equation*}
    \item [(ii)] \textbf{(Lower bound)} For any open subset $\mathcal{O}$ of $\mathcal{D}_{\mathcal{M}}^T$,
    \begin{equation*}
        \underset{N \rightarrow \infty}{\underline{\lim}} ~ \frac{1}{N}\log \mathbb{P}_{\delta_{\e^N}}\left[\mathcal{O} \right] \geq - \underset{\pi \in \mathcal{O}}{\inf}~ \Tilde{I}_T(\pi).
    \end{equation*}
\end{itemize}
    \end{thm}

\subsection*{Main results}
    
To observe an evolution of the total mass of the process
$(\e_t)_{t\geq 0}$, time has to be accelerated by a factor $N$, that
is, by speeding-up the exclusion part by $N^3$ and the boundary
dynamics by $N$.

Fix one and for all a time horizon $T>0$ and denote by
$\{\zeta_t,~ t\in [0,T]\}$ the Markov process with generator 
$$ \mathfrak{L}_N = N^3 \mathcal{L}_{N,0} + N \mathcal{L}_{N,b}= N\LL_N.$$  
We will often refer to $\{\zeta_t,~ t\in [0,T]\}$ as the accelerated process. Given a measure $\mu_N$ on $\Omega_N$, $\Tilde{\mathbb{P}}_{\mu_N}$ is
the probability measure induced on $\mathcal{D}_{\Omega_N}^T$ by the speeded up process
$\{\zeta_t,~ t\in [0,T]\}$ when $\zeta_0$ has law $\mu_N$. Also,
denote by $\Tilde{\mathbb{E}}_{\mu_N}$ the expectation with respect to
$\Tilde{\mathbb{P}}_{\mu_N}$. 

For $\pi \in \mathcal{M}$, introduce
$$
\widehat{m}(\pi) =  \langle1,\pi\rangle = \pi([0,1])
$$
the total mass of $\pi$. Then, the process $\left(\widehat{m}(\pi^N(\zeta_t,.)\right)_{0\leq t\leq T}$, which we will denote by $\left(\widehat{m}(\pi_t^N)\right)_{0\leq t\leq T}$ defines a hidden Markov process with state space $[0,1]$ and induces a probability measure on $\mathcal{D}_{[0,1]}^T$, the space of trajectories defined on $[0,T]$ that are right continuous with left limits and taking their values in $[0,1]$.

The results established in this paper are given in the following subsections.

\subsubsection*{Hydrodynamic limits}

\begin{thm}\label{LGNprocess}(Hydrodynamic limit for the accelerated
process). Fix a measurable profile $\rho_0:[0,1] \rightarrow [0,1]$
and consider a sequence $(\mu_N)_{N \geq 1}$ associated to
$\rho_0$. For any $t\in (0,T]$, $\delta>0$ and
$H\in \mathcal{C}^{0}([0,1])$,
    \begin{equation*}
        \underset{N\rightarrow\infty}{\lim}\Tilde{\mathbb{P}}_{\mu_N}~ \left[\Big|\langle\pi_t^N,H\rangle  - m(t)\langle1,H\rangle \Big|>\delta \right] = 0,
    \end{equation*}
     where $m:[0,T] \rightarrow [0,1]$ is the unique solution of 
 \begin{equation} \label{EDPmasse}
   \left\{
    \begin{array}{ll}
        \partial_tm = -\, 2\, (m -\gamma) \\ 
        m(0)=\int_0^1\rho_0(x)dx,
        \end{array}
\right.
\end{equation}
where $\gamma:= (\alpha+ \beta)/2$.
\end{thm}
For $N\geq 1$ fixed, both Markov processes $(\e_t)_{0\leq t \leq T}$ and $(\zeta_t)_{0\leq t \leq T}$ are irreducible with finite state space and their generators are proportional. They therefore admit a unique same stationary measure on $\Omega_N$ that we denote by $\mu_{ss}^N$.

\begin{thm} \label{LGNmeseq}(Law of large numbers for the invariant measures). For any $H\in \mathcal{C}^0([0,1])$,
\begin{equation*}
    \underset{N \rightarrow \infty}{\lim}~\mathbb{E}_{\mu_{ss}^N}\Big[~ \Big|\langle\pi^N,H \rangle - \gamma\langle 1,H \rangle\Big|~ \Big] =0.
\end{equation*}
\end{thm}

\subsubsection*{Large deviations principles}

We start by defining the rate function that will appear in the dynamical large deviations principle.
\begin{Def}
    For $T>0$ fixed and $G\in \mathcal{C}^1([0,T])$, define $J_{T,G}:\mathcal{D}_{[0,1]}^T \rightarrow \R $ by
    \begin{equation}\label{fonctA}
    \begin{split}
            J_{T,G}(a) =: a_T G_T - a_0 G_0 - \int_0^T \partial_s G_s a_s ds - \int_0^TA_G(a)(s)ds,
\end{split}
\end{equation}
where
\begin{equation}\label{AG}
    A_G(a) := 2\gamma (1-a)\big(e^G-1\big) + 2\big(1-\gamma\big)a \big(e^{-G}-1\big).
\end{equation}
\end{Def}
The rate function is defined as follows.
\begin{Def}\label{defGD} Define $I_T(.):\mathcal{D}_{[0,1]}^T \rightarrow \R \cup \{+\infty\}$ by $I_{T}\big(a\big) = \underset{G\in \mathcal{C}^1([0,T])}{\sup}J_{T,G}(u) $ and for $m\in [0,1]$,
\begin{equation*}
  I_{T}\big(a|m\big) = \left\{
    \begin{array}{ll}
        I_{T}\big(a\big)~ ~ \text{if}~ ~ a(0)=m\\ 
        +\infty ~~~~~~~~~~~ \text{otherwise.}
        \end{array}
\right.
\end{equation*}
\end{Def}

 We define $V:[0,1] \rightarrow [0,+ \infty]$ the quasi potential for the rate function $I_T(~.~|\gamma)$:
\begin{equation*}
    V(m) := \underset{T>0}{\inf} ~ ~ \underset{a(.),~ a(T) = m}{\inf}~ I_T\left(a|\gamma\right),
\end{equation*}
where the infimum is taken over elements of $\mathcal{C}^1([0,T])$.

\begin{lem}\label{Identific}
    The quasi potential satisfies:
    $$\forall m\in [0,1],~ V(m)=S(m), $$
    where
    \begin{equation} \label{EntropieS}
    S(m) := m\log\left(\frac{m}{\gamma} \right) + (1-m) \log\left(\frac{1-m}{1-\gamma} \right).
\end{equation}
\end{lem}
We are now in position to state the dynamic and static large deviations principle.
\begin{thm}\label{DynamicGD}(Dynamical large deviations).
    Fix $m_0\in (0,1)$ an initial mass and consider a sequence of
    configurations $(\e^N)_{N\geq 1}$ such that $
    \widehat{m}(\pi^N(\e^N)) $ converges to $m_0$. The sequence of probability measures $\Tilde{\mathbb{P}}_{\delta_{\e^N}}$ satisfies the following large deviations principle:
    \begin{itemize}
        \item [(i)] For any closed subset $\mathcal{F}$ of $\mathcal{D}_{[0,1]}^T$,
        \begin{equation*}
            \underset{N \rightarrow \infty}{\lim} ~ \frac{1}{N} \log~ \Tilde{\mathbb{P}}_{\delta_{\e^N}}\left[\left(\widehat{m}(\pi_t^N)\right)_{0\leq t\leq T}\in \mathcal{F} \right] \leq - \underset{a\in \mathcal{F}}{\inf} ~ I(a|m_0).
        \end{equation*}
        \item[(ii)] For any open subset $\mathcal{O}$ of $\mathcal{D}_{[0,1]}^T$,
        \begin{equation*}
            \underset{N \rightarrow \infty}{\lim} ~ \frac{1}{N} \log~ \Tilde{\mathbb{P}}_{\delta_{\e^N}}\left[\left(\widehat{m}(\pi_t^N)\right)_{0\leq t\leq T}\in \mathcal{O} \right] \geq - \underset{a\in \mathcal{O}}{\inf} ~ I(a|m_0).
        \end{equation*}
    \end{itemize}
\end{thm}
\begin{thm}\label{StaticGD}(Static large deviations principle).
    \begin{itemize}
        \item[(i)] For any closed subset $\mathscr{F}$ of $[0,1]$,
        \begin{equation} \label{upper}
            \underset{N\rightarrow \infty}{\overline{\lim}}~ \frac{1}{N} \log \mu_{ss}^N\left[\widehat{m}(\pi^N)\in \mathscr{F} \right] \leq -\underset{m\in \mathscr{F}}{\inf} V(m).
        \end{equation}
        \item[(ii)] For any open subset $\mathscr{O}$ of $[0,1]$,
         \begin{equation}
            \underset{N\rightarrow \infty}{\underline{\lim}}~ \frac{1}{N} \log \mu_{ss}^N\left[\widehat{m}(\pi^N)\in \mathscr{O} \right] \geq -\underset{m\in \mathscr{O}}{\inf} V(m).
        \end{equation}
    \end{itemize}
\end{thm}

We prove Theorem \ref{LGNprocess} in section 3, Theorem \ref{DynamicGD} in Section 5
and Theorem \ref{StaticGD} in Section 7.

\section{Hydrodynamic limits}

\subsection*{Dynamical law of large numbers for the accelerated
process}

To prove Theorem \ref{LGNprocess}, which is a pointwise hydrodynamic
result, we first establish the hydrodynamic limit of the total mass
(Proposition \ref{LGNmass}). The latter is also stated pointwisely but
we prove that the convergence of the mass trajectory holds in law,
which yields the pointwise convergence. For that, we follow the standard
steps. First we prove tightness of the sequence of probability
measures induced by the mass process (Lemma \ref{tightness}). Then, we
perform a superexponential replacemment lemma (Lemma \ref{RL}) which
will also be used for the proof of the large deviations principle. For
that, we use some Dirichlet estimates (Lemma \ref{Ineg}).

Fix $\rho\in (0,1)$ and define
$$\mathbb{D}_N(f) = \int_\zeta\sum_{x=1}^{N-2}\left( \sqrt{f(\zeta^{x,x+1})} - \sqrt{f(\zeta)}  \right)^2d\nurhon(\zeta), $$
and
$$D_{N,b}(f) = \int_{\zeta} \sum_{x\in\{1,N-1\}}\left[r_x(1-\zeta(x))+(1-r_x)\zeta(x) \right]\left(\sqrt{f(\zeta^x)}-\sqrt{f(\zeta)} \right)^2d\nurhon(\zeta),$$
for $f:\Omega_N \rightarrow \R$. In this formula, $\nurhon$ represents the
Bernoulli product measure on $\Omega_N$ with density $\rho$.

Fix a constant $m_0\in [0,1]$. We say that a sequence of probability measures $(\mu_N)_{N\geq 1}$ on $\Omega_N$ is associated to the mass $m_0$ if, for any $\delta>0$,
\begin{equation*}
    \underset{N \rightarrow \infty}{\lim} ~\mu_N\left[ \left|\widehat{m}(\pi^N) -m_0 \right|>\delta \right]=0.
\end{equation*}
\begin{prop} \label{LGNmass}(Hydrodynamic limit for the total mass).
Fix $m_0\in [0,1]$ and consider a sequence of measures $(\mu_N)_{N\geq 1}$ on $\Omega_N$ associated to the mass $m_0$. Then, for any $t\in [0,T]$ and any $\delta> 0$,
    $$\underset{N\rightarrow \infty}{\lim}~ \Tilde{\mathbb{P}}_{\mu_N}\left[~ \Big|\widehat{m}(\pi_t^N) -m(t)\Big|> \delta \right] =0, $$
 where $m:[0,T] \rightarrow [0,1]$ is the unique solution of \eqref{EDPmasse} with $m(0)=m_0$.
 \end{prop}
\begin{lem} \label{tightness}
    For any sequence of measures $(\mu_N)_{N\geq 1}$ on $\Omega_N$, the sequence of probability measures $(\Tilde{\mathbb{P}}_{\mu_N})_{N\geq 1}$ induced by $\widehat{m}(\pi_t^N)$ when $\e_0 \sim \mu_N $, is tight for the Skorohod topology. Moreover, all limit points are concentrated on continuous paths.
\end{lem}
\begin{proof}
It is enough to show that for any $\varepsilon>0$,
\begin{equation}\label{tight}
    \underset{\delta \rightarrow 0}{\lim}~ \underset{N \rightarrow \infty}{\limsup}~ \Tilde{\mathbb{P}}_{\mu_N}\left[ \underset{|t-s|\leq \delta}{\sup}~ |\widehat{m}(\pi_t^N) -  \widehat{m}(\pi_s^N)|>\varepsilon\right]=0.
\end{equation}
By Dynkin's formula (see \cite[Appendix 1]{KL}),
\begin{equation*}\label{Dynk}
    M_t^N= \widehat{m}(\pi_t^N) - \widehat{m}(\pi_0^N) - N^3\inte \mathcal{L}_{N,0}\widehat{m}(\pi_s^N)\mathrm{d}{s} - N\inte \mathcal{L}_{N,b}\widehat{m}(\pi_s^N)\mathrm{d}{s}
\end{equation*}
is a martingale with respect to the natural filtration $\mathcal{F}_t:=\sigma(\e_s,~ s\leq t)$. As the bulk dynamics is conservative,
$$\mathcal{L}_{N,0}\widehat{m}(\pi_s^N)=0.$$
Also, computations yield:
$$N\mathcal{L}_{N,b}\widehat{m}(\pi_s^N) = \alpha - \zeta_s(1) + \beta - \zeta_s(N-1),  $$
so
\begin{equation*}
    M_t^N= \widehat{m}(\pi_t^N) - \widehat{m}(\pi_0^N) - \inte(\alpha + \beta -\zeta_s(1)-\zeta_s(N-1))\mathrm{d}{s}.
\end{equation*}
Therefore, \eqref{tight} holds if
\begin{equation}\label{tight1}
    \underset{\delta \rightarrow 0}{\lim}~ \underset{N \rightarrow \infty}{\limsup}~\Tilde{\mathbb{E}}_{\mu_N} \left[\underset{|t-s|\leq \delta}{\sup}~\Big|M_t^N-M_s^N \Big| \right] = 0
\end{equation}
and
\begin{equation}\label{tight2}
    \underset{\delta \rightarrow 0}{\lim}~ \underset{N \rightarrow \infty}{\limsup}~ \Tilde{\mathbb{E}}_{\mu_N}\left[\underset{|t-s|\leq \delta}{\sup}~\Big|\int_s^t \big(\alpha + \beta - \zeta_r(1) - \zeta_r(N-1)\big)\mathrm{d}{r}  \Big| \right] = 0.
\end{equation}
Denote by $\langle M_t^N\rangle$ the quadratic variation of $M_t^N$. By Doob's inequality,
$$\Tilde{\mathbb{E}}_{\mu_N} \left[\underset{|t-s|\leq \delta}{\sup}~\Big|M_t^N-M_s^N \Big| \right] \leq 2\Tilde{\mathbb{E}}_{\mu_N} \left[\underset{0\leq t\leq T}{\sup}~\Big|M_t^N\Big| \right] \leq 4 \Tilde{\mathbb{E}}_{\mu_N}\left[\langle M_T^N\rangle \right]^{1/2}. $$
Dynkin's formula and the fact that the bulk dynamics is conservative yields the following expression for $\langle M_t^N\rangle $:
\begin{equation}
\label{Varquadra}
\begin{split}
\langle M_t^N\rangle &= N \int_0^t\left[\mathcal{L}_{N,b}\widehat{m}(\pi_s^N) ^2-2\widehat{m}(\pi_s^N) \mathcal{L}_{N,b}\widehat{m}(\pi_s^N)  \right]\mathrm{d}{s}\\
&=  \frac{1}{N}\inte \sum_{x\in \{1,N-1\}}
\left[r_x\zeta_s(x)+(1-r_x)(1-\zeta_s(x)) \right] \, ds\;.
\end{split}
\end{equation} 
Hence, as $\zeta_s(x)$ is bounded,
$\langle M_t^N\rangle \underset{N \rightarrow \infty}{
\longrightarrow} 0, $ as claimed in \eqref{tight1}. To prove
\eqref{tight2}, simply note that for $0\leq s \leq t$,
$N\mathcal{L}_{N,0}\widehat{m}(\pi_s^N) = \alpha + \beta
-\zeta_s(1)-\zeta_s(N-1)$ is uniformly bounded.
\end{proof}

\begin{lem}(Superexponential Replacemment Lemma). 
\label{RL}
Given $g\in \mathcal{C}([0,T])$, for any sequence of probability
measures $\mu_N$ on $\Omega_N$, $t\in [0,T]$, $\delta>0$ and
$x\in \{1,N-1\}$,
    $$ \underset{N\rightarrow \infty}{\overline{\lim}}~  \frac{1}{N} \log ~ \Tilde{\mathbb{P}}_{\mu_N}\left[~ \Big|\inte g(s)(\zeta_s(x)-\widehat{m}(\pi_s^N) )\mathrm{d}{s} \Big|>\delta \right] = -\infty.$$
\end{lem}

For the proof of that, we use the following :

\begin{lem}\label{Ineg}
    For any density $f$ with respect to $\nurhon$,
    \begin{itemize}
        \item [(i)] $\langle\mathcal{L}_{N,0}\sqrt{f},\sqrt{f}\rangle _{\nurhon} = -\frac{1}{2}\mathbb{D}_N(f)$
        \item [(ii)] $\langle\mathcal{L}_{N,b}\sqrt{f},\sqrt{f}\rangle _{\nurhon} = -\frac{1}{2}D_{N,b}(f) + U_N  $ where $(U_N)_{N\geq 1}$ is a uniformly bounded sequence.
    \end{itemize}
\end{lem}

\begin{proof}
\noindent
    \begin{itemize}
    \item [(i)] This point comes from the invariance of $\nurhon$ under the dynamics generated by $\mathcal{L}_{N,0}$ and a change of variables.
    \item[(ii)] Developing the term $\langle\mathcal{L}_{N,b}\sqrt{f},\sqrt{f}\rangle _{\nurhon} $ and performing the change of variables $\zeta \rightarrow \zeta^x$ we get that
    \begin{align*}
        \langle\mathcal{L}_{N,b}\sqrt{f},\sqrt{f}\rangle _{\nurhon} &= -\frac{1}{2}D_{N,b}(f) +  \frac{1}{2}\int\sum_{x\in \{1,N-1\}}\left[r_x\zeta(x)+(1-r_x)(1-\zeta(x)) \right]f(\zeta)\left(\frac{\rho}{1-\rho} \right)^{1-2\zeta(x)}d\nurhon(\zeta).
    \end{align*}
     Now, using that $f$ is a density, that the jump rates are bounded and that $\rho\in (0,1)$, the result follows.
\end{itemize}
\end{proof}
\textit{Proof of Lemma \ref{RL}.} Fix $\rho \in (0,1)$ and $x\in \{1,N-1\}$ . For $s\in [0,t]$ and $g\in \mathcal{C}([0,T])$, write $V_s^g(\zeta) =g(s)(\zeta_s(x)-\widehat{m}(\pi_s^N) )$. Using that $\Big|\frac{d\mu_N}{d\nurhon}\Big|\leq e^{CN}$
and Tchebychev's inequality, 
\begin{equation}\label{Tcheby}
   \Tilde{\mathbb{P}}_{\mu_N}\left[~ \Big|\inte V_s^g(\zeta)\mathrm{d}{s} \Big|>\delta \right]\leq \exp(-(c\delta-C) N)\times \Tilde{\mathbb{E}}_{\nurhon}\left[ \exp\left( cN\Big|\inte V_s^g(\zeta)\mathrm{d}{s}\Big|\right)\right],
\end{equation}
for any $c>0$.
Therefore, 
\begin{equation} \label{Bigproba}
    \frac{1}{N} \log~ \Tilde{\mathbb{P}}_{\mu_N}\left[~ \Big|\inte V_s^g(\zeta)\mathrm{d}{s} \Big|>\delta \right] \leq C-c\delta + \frac{1}{N}\log ~ \Tilde{\mathbb{E}}_{\nurhon}\left[ \exp\left( cN\Big|\inte V_s^g(\zeta)\mathrm{d}{s}\Big|\right)\right].
\end{equation}
Let us deal with the last term. As $e^{|x|}\leq e^x + e^{-x},$ and
\begin{equation}\label{loginequ}
    \underset{N\rightarrow\infty}{\overline{\lim}}\frac{1}{N}\log(a_N+ b_N) \leq \max\left(\underset{N\rightarrow\infty}{\overline{\lim}}\frac{1}{N}\log a_N,\underset{N\rightarrow\infty}{\overline{\lim}}\frac{1}{N}\log b_N \right), 
\end{equation}
we bound that term without the absolute values. 

Let $V^N(\zeta)= \zeta(x)-\langle 1, \pi^N\rangle$.  By Feynman-Kac's
inequality, stated, for instance in \cite[Appendix 1]{KL}, Lemma
\ref{Ineg} and the fact that $D_{N,b}$ is non negative,
\begin{equation}\label{Bigexpp}
    \begin{split}
        &\Tilde{\mathbb{E}}_{\nurhon}\left[ \exp\left( cN\inte  V_s^g(\zeta)\mathrm{d}{s}\right)\right]\leq \exp\left(\inte \underset{f}{\sup}\Big\{\int cNg(s)V^N(\zeta)f(\zeta)d\nurhon(\zeta) + N^3\langle\mathcal{L}_{N,0}\sqrt{f},\sqrt{f} \rangle  +N\langle\mathcal{L}_{N,b}\sqrt{f},\sqrt{f} \rangle  \Big\}ds \right)\\
        & \leq \exp\left(\inte \underset{f}{\sup}\Big\{\int cN g(s)V^N(\zeta)f(\zeta)d\nurhon(\zeta) - \frac{N^3}{2}\mathbb{D}_N(f)+ N U_N \Big\}ds \right),
    \end{split}
\end{equation}
where the supremum is carried over densities with respect to
$\nurhon$. Without loss of generality, suppose $x=N-1$. Then,
\begin{align*}
    \int_{\zeta}V^N(\zeta)f(\zeta)d\nurhon(\zeta)&= \frac{1}{N}\sum_{y=1}^{N-1}\int_{\zeta}(\zeta(N-1)-\zeta(y))f(\zeta)d\nurhon(\zeta)\\
    &= \frac{1}{N}\sum_{y=1}^{N-1} \sum_{i=y}^{N-2}\int \big(\zeta(i+1)-\zeta(i)\big)f(\zeta)d\nurhon (\zeta)\\
    &= \frac{1}{2N}\sum_{y=1}^{N-1}\sum_{i=y}^{N-2}\int \big(\zeta(i+1)-\zeta(i)\big)\big(f(\zeta)-f(\zeta^{i,i+1})\big)d\nurhon (\zeta),
\end{align*}
where we performed the change of variable $\zeta \rightarrow \zeta^{i,i+1}$ in the last line. Now, write
$$f(\zeta)-f(\zeta^{i,i+1}) =\big(\sqrt{f(\zeta)} -\sqrt{f(\zeta^{i,i+1})}\big)\big(\sqrt{f(\zeta)} +\sqrt{f(\zeta^{i,i+1})} \big).  $$
By Young's inequality and using the fact that $\zeta$ is bounded, the above is less than
\begin{align*}
    &\frac{B}{4N}\sum_{y=1}^{N-1}\sum_{i=y}^{N-2}\int \left(\sqrt{f(\zeta^{i,i+1})}-\sqrt{f(\zeta)} \right)^2d\nurhon(\zeta) + \frac{1}{4NB}\sum_{y=1}^{N-1}\sum_{i=y}^{N-2}\int \left(\sqrt{f(\zeta^{i,i+1})}+\sqrt{f(\zeta)} \right)^2 d\nurhon(\zeta)\\
    &\leq \frac{B}{4N}\sum_{y=1}^{N-1}\mathbb{D}_N(f) + \frac{N}{B},
\end{align*}
for any $B>0$, where we used that $f$ is a density. Taking the log of \eqref{Bigexpp}, dividing it by $N$ and using the above inequality yields
\begin{equation*}
    \begin{split}
        \frac{1}{N}\log ~ \Tilde{\mathbb{E}}_{\nurhon}\left[ \exp\left( cN\inte V_s^g(\zeta)\mathrm{d}{s}\right)\right] &\leq \int_0^t \underset{f}{\sup} \Big\{\frac{c|g(s)|B}{4}\mathbb{D}_N(f) + \frac{cN}{B}|g(s)| - \frac{N^2}{2}\mathbb{D}_N(f) + U_N \Big\}ds\\
        &\leq T ~ \underset{f}{\sup} \Big\{\frac{c\|g\|_{\infty}B}{4}\mathbb{D}_N(f) + \frac{cN}{B}\|g\|_{\infty} - \frac{N^2}{2}\mathbb{D}_N(f) + U_N \Big\}\\
        &\leq T ~ \Big\{\frac{c^2\|g\|^2_{\infty}}{N}+ U_N \Big\},
    \end{split}
\end{equation*}
where to get the last inequality we took $B= \frac{2N^2}{c\|g\|_{\infty}}$. As $U_N$ is uniformly bounded by a constant $C_0$, we are left with
\begin{equation*}
    \lims \frac{1}{N}\log ~ \Tilde{\mathbb{E}}_{\nurhon}\left[ \exp\left( cN\inte V_s^g(\zeta)\mathrm{d}{s}\right)\right] \leq TC_0
\end{equation*}
and by \eqref{Bigproba},
\begin{equation*}
    \lims \frac{1}{N} \log~ \Tilde{\mathbb{P}}_{\nurhon}\left[~
    \Big|\inte V_s^g(\zeta)\mathrm{d}{s} \Big|>\delta \right]  \leq C-
    c \delta + TC_0 .
\end{equation*}
Taking $c\rightarrow +\infty$ yields the result.\\\\

\textit{Proof of Proposition \ref{LGNmass}.} 
From \eqref{Varquadra}, we have that $$\Tilde{\mathbb{E}}_{\mu_N}\left[~ \underset{0\leq t \leq T}{\sup}~ \Big|M_t^N \Big|~ \right] \underset{N\rightarrow \infty}{\longrightarrow} 0,$$
and using Lemma \ref{RL}, for any $\varepsilon>0$, for any $t\in [0,T]$
$$\Tilde{\mathbb{P}}_{\mu_N}\left[~  \Big|\widehat{m}(\pi_t^N) - \widehat{m}(\pi_0^N) - \inte 2(\gamma - \widehat{m}(\pi_s^N))\mathrm{d}{s} \Big|> \varepsilon ~ \right] \underset{N\rightarrow \infty}{\longrightarrow} 0.$$
Since all limit points of the sequence $(\Tilde{\mathbb{P}}_{\mu_N})_{N\geq 1}$ are concentrated on continuous trajectories,
by PortManteau's Lemma, for any limit point $\mathbb{Q}^*$ of $(\Tilde{\mathbb{P}}_{\mu_N})_{N\geq 1}$, for any $\varepsilon>0$ and $t\in [0,T]$,
\begin{equation}\label{Portmant}
    \mathbb{Q}^*\left[~\Big|\widehat{m}(\pi_t) - \widehat{m}(\pi_0) -\inte 2(\gamma - \widehat{m}(\pi_s))\mathrm{d}{s}\Big| > \varepsilon \right]=0.
\end{equation}
Taking $\varepsilon \rightarrow 0$, for any $t\in [0,T]$,
\begin{equation}
    \mathbb{Q}^*\left[~\widehat{m}(\pi_t) - \widehat{m}(\pi_0) =\inte 2(\gamma  - \widehat{m}(\pi_s))\mathrm{d}{s} \right]=1.
\end{equation}
Taking a dense subset of times and using the
right-continuity of the trajectories yields that 
\begin{equation}\label{limitpoint}
    \mathbb{Q}^*\left[~\forall t\in [0,T]~ \widehat{m}(\pi_t) - \widehat{m}(\pi_0) =\inte 2(\gamma- \widehat{m}(\pi_s))\mathrm{d}{s} \right]=1.
\end{equation}
If $(\mu_N)_{N\geq 1}$ is associated to a mass $m_0$,
$$\mathbb{Q}^*\left[~ \forall t\in [0,T]~ \widehat{m}(\pi_t)  =m_0+\inte 2(\gamma- \widehat{m}(\pi_s))\mathrm{d}{s} \right]=1, $$
so $\mathbb{Q}^*$ is the unique measure concentrated on $\delta_{m(t)}$ where $m$ is the unique solution of \eqref{EDPmasse} with $m(0)=m_0$. Convergence in law to a deterministic measure implies convergence in probability so we proved Proposition \ref{LGNmass}. Note that $m$ is explicitly given by
\begin{equation}
\label{f01}
    m(t) = \gamma + \left( m_0-\gamma\right)e^{-2t}.
    \end{equation}
    This completes the proof of the proposition. \qed
    
 Now, we move on to the proof of Theorem \ref{LGNprocess}, the hydrodynamic limit of the empirical measure associated to the accelerated process. The idea is to use the fact that the hydrodynamic limit of the non accelerated process is the heat equation with Neumann boundary conditions (see Theorem \ref{LGNnonacc}) and that the solution of that equation converges in time to a flat profile given by the total mass of the initial condition (Lemma \ref{analyse}). The accelerated process then only evolves according to the hydrodynamic limit of the mass, stated in Proposition \ref{LGNmass}. This is the object of Proposition \ref{hydrossep}.
\begin{prop}\label{hydrossep} Consider $(\mu_N)_{N\geq 1}$ a sequence of probability measures, fix $a>0$ and $H\in \mathcal{C}^{0}([0,1])$. There is a $t_0>0$, depending on $a$ and $H$ such that
\begin{equation} \label{Proba}
    \underset{N\rightarrow \infty}{\overline{\lim}} ~ \Tilde{\mathbb{P}}_{\mu_N}\left[~ \Big|\langle\pi_\frac{t}{N}^N,H \rangle -\widehat{m}(\pi_0^N)\langle1,H\rangle  \Big|>a \right] = 0,
\end{equation}
for any $t\geq t_0$.
\end{prop}
To prove Proposition \ref{hydrossep} we use the following Lemma for which a proof can be found for instance in \cite{evans_partial_2010} or \cite{PDE}, Chapter 4.
\begin{lem}\label{analyse} Consider a measurable profile
$\rho_0:[0,1]\rightarrow [0,1]$ and let $\rho$ be the solution in
$\mathcal{H}^1(0,1)$ of \eqref{heatNeumann}. Then, for any $t\geq 0$,
$$\|\rho(t,.)-\int_0^1\rho_0(x)dx\|_{L^2}^2 \leq \|\rho(0,.)-\int_0^1\rho_0(x)dx\|_{L^2}^2e^{-2\lambda_1t}\leq 2e^{-2\lambda_1t},  $$
where $\lambda_1$ is the smallest non zero eigenvalue of the Laplacian associated to Neumann boundary conditions.
\end{lem}
\noindent \textit{Proof of Proposition \ref{hydrossep}}.

Fix $a>0$, and $H\in \mathcal{C}^{0}([0,1])$. Consider $t_0>0$ such
that $2 e^{-\lambda_1 t_0}\|H\|_{L^2(0,1)}<\frac{a}{2}$ and fix
$t\geq t_0$. Let $(N_k)_{k\geq 1}$ be a sequence along which the
probability \eqref{Proba} converges. Let us show that the limit is
necessarily zero. It is enough to prove that there is a subsequence of
$(N_k)_{k\geq 1}$ such that the limit is zero.

Define $\mathcal{A}_{t} \subset\mathcal{D}_{\mathcal{M}_0}^t$ as the
set of trajectories $\{\rho(r,\cdot),~ 0\leq r\leq t\}$ whose density
$\rho$ is a solution of \eqref{heatNeumann} for some initial condition. The process
$(\pi_\frac{r}{N}^N)_{0\leq r \leq t}$ has generator $\mathcal{L}_N$
and by the proof of Theorem \ref{LGNnonacc} in
\cite{SSEPslowboundary}, the sequence of probability measures
$\Tilde{\mathbb{P}}_{\mu_N}$ on $\mathcal{D}_{\mathcal{M}}^t$ induced
by $(\pi^N_{\frac{r}{N}})_{0\leq r \leq t}$ when $\pi^N_0 \sim \mu_N$,
is tight and all its limit points $\mathbb{P}^*$ satisfy
$$\mathbb{P}^*\left(\mathcal{A}_{t} \right)=1. $$
Consider $(N_j)_{j\geq 1}$ a subsequence of $(N_k)_{k\geq 1}$ such that $\Tilde{\mathbb{P}}_{\mu_{N_j}}$ converges to a probability measure $\mathbb{P}^*$. Then,
\begin{equation}\label{proba2}
\begin{split}
        \underset{j \rightarrow \infty}{\overline{\lim}}~ \Tilde{\mathbb{P}}_{\mu_{N_j}}\left[ \Big|\langle\pi_{\frac{t}{N_j}}^{N_j},H\rangle - \widehat{m}(\pi_0^{N_j})\langle 1,H\rangle \Big|>a \right] &= \mathbb{P}^* \left[\left(\Big|\langle\rho(t),H\rangle - \int_0^1 \rho(0,x)dx \langle 1, H \rangle\Big|>a  \right) \bigcap \left(\rho\in  \mathcal{A}_{t}\right)\right].
\end{split}
\end{equation}
By Lemma \ref{analyse}, for any $\rho\in \mathcal{A}_{t} $,
\begin{align*}
    |\langle\rho_t,H\rangle -\int_0^1 \rho(0,x)dx \langle1,H\rangle |& \leq \|\rho(t,.)-\int_0^1 \rho(0,x)dx\|_{L^2(0,1)}\|H\|_{L^2(0,1)}\\
    &\leq 2 e^{-\lambda_1 t}\|H\|_{L^2(0,1)}.
\end{align*}
Therefore, for any $t\geq t_0$ and $\rho\in \mathcal{A}_{t} $
\begin{align*}
    |\langle\rho_t,H\rangle -\int_0^1 \rho(0,x)dx \langle1,H\rangle |& < \frac{a}{2},
\end{align*}
which implies that the limit in \eqref{proba2} is necessarily zero. \\\\
\textit{Proof of Theorem \ref{LGNprocess}.} 
Fix $H\in \mathcal{C}^{0}([0,1])$. The idea is to split the term $|\langle\pi_t^N,H\rangle -m(t)\langle1,H\rangle |$ into two parts involving the total mass process $\widehat{m}(\pi_t^N)$. One part is dealt with thanks to the hydrodynamic limit for the total mass given in Theorem \ref{LGNmass}. The other part relies on Proposition \ref{hydrossep}.

Fix a $t\in (0,T]$. We have
\begin{equation*}
    \begin{split}
        |\langle\pi_t^N,H\rangle -m(t)\langle1,H\rangle |  &\leq |\langle\pi_t^N,H\rangle -\widehat{m}(\pi_t^N) \langle1,H\rangle  |\\
        &+ |\widehat{m}(\pi_t^N) -m(t)| |\langle1,H\rangle |.
    \end{split}
\end{equation*}
By Theorem \ref{LGNmass}, the second term above converges in probability to zero under $\Tilde{\mathbb{P}}_{\mu_N}$. Let us prove that 
\begin{equation*}
    \underset{N\rightarrow \infty}{\lim}~ \Tilde{\mathbb{P}}_{\mu_N} \left[~ \Big| \langle\pi_t^N,H\rangle -\widehat{m}(\pi_t^N)\langle 1,H\rangle \Big|>\varepsilon \right] = 0
\end{equation*}
for all $\varepsilon>0$. Consider $t_0$ such that
 $$2 e^{-\lambda_1 t_0}\|H\|_{L^2(0,1)}<\frac{\varepsilon}{4},$$
where we recall that $\lambda_1$ was introduced in Lemma \ref{analyse}. Also, consider $N$ large enough so that $t-\frac{t_0}{N}\geq 0$. We have
\begin{equation*}
    \begin{split}
        \big|\langle\pi_t^N,H\rangle -\widehat{m}(\pi_t^N) \langle1,H\rangle  \big|& \leq \left|\langle\pi_t^N,H\rangle -\widehat{m}\left(\pi_{t-\frac{t_0}{N}}^N\right)\langle 1,H\rangle  \right|\\
        &+ \left|\widehat{m}\left(\pi_{t-\frac{t_0}{N}}^N\right)  - \widehat{m}(\pi_t^N) \right| \left|\langle1,H\rangle \right|.
    \end{split}
\end{equation*}
By Lemma \ref{tightness},
$$\underset{N\rightarrow \infty}{\lim}~ \Tilde{\mathbb{P}}_{\mu_N}\left[~ \Big|\widehat{m}\left(\pi_{t-\frac{t_0}{N}}^N\right)  - \widehat{m}(\pi_t^N)\Big| \big|\langle1,H\rangle \big|>\frac{\varepsilon}{2} \right]=0. $$
Now, we are left to show that
\begin{equation*}
   \underset{N\rightarrow \infty}{\lim}~ \Tilde{\mathbb{P}}_{\mu_N}\left[~ \left|\langle\pi_t^N,H\rangle -\widehat{m}\left(\pi_{t-\frac{t_0}{N}}^N\right) \langle 1,H\rangle  \right|>\frac{\varepsilon}{2} \right]=0.
\end{equation*}
Consider the process $\left(\Tilde{\pi}_s^N:=\pi^N_{t-\frac{t_0}{N}+ \frac{s}{N} }\right)_{s\geq 0}$ so that $\Tilde{\pi}_{0}^N = \pi_{t-\frac{t_0}{N}}^N$ and $\Tilde{\pi}_{t_0}^N = \pi_t^N$. As the process is Markovian, 
\begin{align*}
    \Tilde{\mathbb{P}}_{\mu_N}\left[~ \left|\langle\pi_t^N,H\rangle -\widehat{m}\left(\pi_{t-\frac{t_0}{N}}^N\right)  \langle1,H\rangle \right|>\frac{\varepsilon}{2} \right] = \Tilde{\mathbb{P}}_{S_{t-\frac{t_0}{N}}(\mu_N)}\left[~  \Big|\langle\Tilde{\pi}_{t_0}^N,H\rangle -\widehat{m}(\Tilde{\pi}_0^N) \langle 1,H\rangle  \Big|>\frac{\varepsilon}{2}\right],
\end{align*}
where $S_{t-\frac{t_0}{N}}(\mu_N)$ is the push-forward of the measure
$\mu_N$ under the dynamics. By Proposition \ref{hydrossep} and the
choice of $t_0$, this goes to zero as $N$ goes to infinity. \qed

\subsection*{Static law of large numbers for the accelerated process}

In order to prove Theorem \ref{LGNmeseq}, we start by proving the
static law of large numbers for the total mass of the process (see
Proposition \ref{LGNequ}). This result follows from the fact that any
solution to the O.D.E \eqref{EDPmasse} of the total mass converges to
$\gamma$. Then, we use a similar argument as in the proof of Theorem
\ref{LGNprocess} which consists in moving by a factor $t/N$ back in
time, to recover the static law of large numbers for the accelerated
process of empirical measures.

Recall that $\mu_{ss}^N$ denotes the unique stationary measure on $\Omega_N$ relatively to the dynamics induced by the generator $\mathcal{L}_N$. 

\begin{prop}\label{LGNequ}
    The sequence of invariant measures $(\mu_{ss}^N)_{N\geq 1}$ satisfies:
    \begin{equation*}
        \underset{N \rightarrow \infty}{\lim}~ E_{\mu_{ss}^N}\left[~ \Big| \widehat{m}(\pi^N)- \gamma\Big| ~\right] =  0,
    \end{equation*} 
    where $E_{\mu_{ss}^N}$ is the expectation under $\mu_{ss}^N$ on $\Omega_N$.
\end{prop}

\begin{proof} As $\Big(\widehat{m}(\pi^N)- \gamma\Big)_{N\geq 1}$ is $\mu_{ss}^N$-almost surely bounded, it is enough to prove the convergence in probability to zero. Fix $\varepsilon>0$ and choose $T=T(\varepsilon)$ such that for any $m_0\in [0,1]$, $\big|m_0-\gamma\big|e^{-2T} \leq \varepsilon$. 

Consider $(N_k)_{k\geq 1}$ a sequence such that
$$\lims~  \mu_{ss}^N\Big(~ \Big|\widehat{m}(\pi^N)- \gamma\Big|>\varepsilon ~ \Big)= \underset{k \rightarrow \infty}{\lim}~ \mu_{ss}^{N_k}\Big(~ \Big|\widehat{m}(\pi^{N_k})- \gamma\Big|>\varepsilon~  \Big).$$
By stationarity of $\mu_{ss}^N$,
$$\mu_{ss}^{N_k}\Big(~ \Big|\widehat{m}(\pi^{N_k})- \gamma\Big|>\varepsilon~  \Big) = \Tilde{\mathbb{P}}_{\mu_{ss}^{N_k}}\Big(~ \Big|\widehat{m}(\pi_T^{N_k})- \gamma\Big|>\varepsilon~  \Big). $$
By Lemma \ref{tightness}, one can extract from $\Tilde{\mathbb{P}}_{\mu_{ss}^{N_k}}$ a converging subsequence $\Tilde{\mathbb{P}}_{\mu_{ss}^{N_j}}$ with limit $\mathbb{P}^*$, so that
\begin{align*}
    \underset{k \rightarrow \infty}{\lim}~ \mu_{ss}^{N_k}\Big(~ \Big|\widehat{m}(\pi^{N_k})- \gamma\Big|>\varepsilon~  \Big) &= \underset{j \rightarrow \infty}{\lim}~\Tilde{\mathbb{P}}_{\mu_{ss}^{N_j}}\Big(~ \Big|\widehat{m}(\pi_T^{N_j})- \gamma\Big|>\varepsilon~  \Big)\\
    &= \mathbb{P}^*\Big(~ \Big|\widehat{m}(\pi_T)- \gamma\Big|>\varepsilon~  \Big)\\
    &= \mathbb{P}^*\Big(~ \Big|\widehat{m}(\pi_0)- \gamma\Big|e^{-2T}>\varepsilon~  \Big) = 0,
\end{align*}
where the last line comes from \eqref{Portmant} and the choice of $T$.
\end{proof}

\begin{rem} \label{R1} One can deduce from the proof of Proposition
\ref{LGNequ} that there is a sequence
$(\varepsilon_N)_{N\geq 1} \downarrow 0$ such that
$\mu_{ss}^N\left(B_{\varepsilon_N} \right)$ converges to $1$, where
$B_{\varepsilon_N}$ is the set of measures with total mass less than
$\varepsilon_N$ away from $\gamma$:
$$B_{\varepsilon_N} = \left\{\pi \in \mathcal{M},~ |\widehat{m}(\pi)-\gamma|<\varepsilon_N \right\}.  $$
\end{rem}

\textit{Proof of Theorem \ref{LGNmeseq}.} Consider $(\pi_t^N)_{t\geq 0}$ the accelerated process of empirical measures, that is, with generator $\mathfrak{L}_N$, and such that $\pi_0^N \sim \mu_{ss}^N$. Fix $G\in \mathcal{C}^0([0,1])$ and $T>0$ which will be taken sufficiently large later on.  As the generator of the accelerated process is proportional to the one for the non accelerated one, $\mu_{ss}^N$ is invariant for that process, therefore,
\begin{equation}
    \begin{split}
        \mu_{ss}^N\Big[~ \Big|\langle \pi^N,G\rangle - \gamma\langle 1,G\rangle \Big|~ \Big] &= \E_{\mu_{ss}^N}\Big[~ \Big|\langle \pi_T^N,G\rangle - \gamma\langle 1,G\rangle \Big|~ \Big]\\
        & \leq \E_{\mu_{ss}^N}\Big[~ \Big|\langle \pi_T^N,G\rangle - \widehat{m}(\pi_{T-\frac{t_0}{N}}^N)\langle 1,G\rangle \Big|~ \Big]~~~~:= a_{N,T,t_0}\\
        & + \E_{\mu_{ss}^N}\Big[~ \Big|\widehat{m}(\pi_{T-\frac{t_0}{N}}^N)-\widehat{m}(\pi_{T}^N) \Big|\langle 1,G\rangle~ \Big]~~~~~~:= b_{N,T,t_0}\\
        & + \E_{\mu_{ss}^N}\Big[~ \Big|\widehat{m}(\pi_{T}^N) -\gamma\Big|\langle 1,G\rangle~ \Big]~~~~~~~~~~~:= c_{N,T}
    \end{split}
\end{equation}
where $N$ is large enough so that $T-\frac{t_0}{N} >0$. By Proposition
\ref{LGNequ}, $c_{N,T} \rightarrow 0$, when $N \rightarrow \infty$ and
then $T\rightarrow \infty$.

To control $b_{N,T,t_0}$, using the proof of Lemma \ref{tightness} we get:
\begin{equation*}
   \lims b_{N,T,t_0} \leq  \lims \E_{\mu_{ss}^N} \Big[~ \underset{|s-t|\leq \frac{t_0}{N}}{\sup}~\Big|M_t^N- M_s^N  \Big|~   \Big] + \lims \E_{\mu_{ss}^N} \Big[~ \underset{|s-t|\leq \frac{t_0}{N}}{\sup}~\Big|N \int_s^t \mathcal{L}_{N,b}\widehat{m}(\pi_r^N)dr \Big|~   \Big] =0.
\end{equation*}
Finally, to control $a_{N,T,t_0}$, introduce the process $\left(\Tilde{\pi}_s^N:=\pi^N_{t-\frac{t_0}{N}+ \frac{s}{N} }\right)_{s\geq 0}$ as in the proof of Theorem \ref{LGNprocess}. By Markov's property,
\begin{equation*}
    a_{N,T,t_0} = \E_{\mu_{ss}^N} \Big[~\Big|\langle \Tilde{\pi}_{t_0}^N ,G\rangle - \langle \Tilde{\pi}_{0}^N ,1\rangle \langle 1,G\rangle\Big| ~ \Big].
\end{equation*}
By Lemma \ref{tightness}, $(\Tilde{\mathbb{P}}_{\mu_{ss}^N})_{N\geq 1}$ is tight so we can extract from it a sequence $(\Tilde{\mathbb{P}}_{\mu_{ss}^{N_k}})_{k\geq 1}$ converging to a probability measure $\mathbb{Q}^*$. By Theorem \ref{LGNnonacc}, 
\begin{equation*}
    \mathbb{Q}^*\Big(\pi \in \mathcal{A}_{t_0} \Big) = 1,
\end{equation*}
where 
$$\mathcal{A}_{t_0} = \Big\{\big(\rho(t,x)dx \big)_{t\in [0,t_0]} ,~ \rho~ \text{is a weak solution of \eqref{heatNeumann}}\Big\}. $$
Therefore,
\begin{equation*}
    \begin{split}
        \lims \E_{\mu_{ss}^{N_k}} \Big[~\Big|\langle \Tilde{\pi}_{t_0}^{N_k} ,G\rangle - \langle \Tilde{\pi}_{0}^{N_k} ,1\rangle \langle 1,G\rangle\Big| ~ \Big] &= \E_{\mathbb{Q}^*} \Big[ \Big|\langle \pi_{t_0},G\rangle - \langle \pi_{0},1\rangle \langle 1,G\rangle  \Big| \mathds{1}_{\pi\in \mathcal{A}_{t_0}}\Big]\\
        & \leq \E_{\mathbb{Q}^*} \Big[ \|\rho(t_0,.) - \int_0^1 \rho(0,x)dx \|_{L^2(0,1)}\|G\|_{L^2(0,1)} \mathds{1}_{\pi\in \mathcal{A}_{t_0}}\Big]\\
        & \leq 2e^{-\lambda_1t_0}\|G\|_{L^2(0,1)} \underset{t_0\rightarrow \infty}{\longrightarrow} 0.
    \end{split}
\end{equation*}
Hence $\underset{t_0 \rightarrow \infty}{\overline{\lim}} \lims a_{N,T,t_0}=0$ and the result follows.

\section{Large deviations functional}
In this section, we list some properties on the large deviations functional $I_T$, introduced in Definition \ref{defGD}, that will be used later on. We prove the $I_T$ - density, used for the proof of the lower bound of Theorem \ref{DynamicGD}, and we prove Lemma \ref{Identific}.

\subsection{Some properties on $I_T$}
It is immediate to see that $I_T$ is lower semi-continuous, as the supremum of linear functions and therefore continuous functions. Furthermore, we show that it is infinite on non continuous trajectories, that it has compact level sets and that its variational formulation is solvable on a particular class of elements of $\mathcal{D}_{[0,1]}^T$.
\begin{prop}\label{Continu}
If $a\in \mathcal{D}_{[0,1]}^T$ satisfies $I_T(a)<\infty$, then $a$ is continuous.
\end{prop}
\begin{proof}
    Consider $0\leq s<t \leq T$ such that $t-s\leq 1$ and let $H_{s,t}:[0,T]\rightarrow \R$ be given by $H_{s,t}(r)= \log(1/(t-s))\mathds{1}_{[s,t]}(r)$. Also, consider a smooth approximation $H_{s,t}^p:[0,T]\rightarrow \R$ of $H_{s,t}$ in the sense that for any element $f$ of $\mathcal{C}^{\infty}([0,T])$,
$$\int_0^T f(r)H_{s,t}^p(r) dr \underset{p\rightarrow \infty}{\longrightarrow}  \int_0^T f(r)H_{s,t}(r) dr,~ ~\text{and}~ ~  \int_0^T f(r) \partial_r H_{s,t}^p(r) dr \underset{p\rightarrow \infty}{\longrightarrow}  \int_0^T f(r) \partial_r H_{s,t}(r) dr $$
where $\partial_r H_{s,t}^p$ resp. $\partial_r H_{s,t}$ refers to the weak derivative of $H_{s,t}^p $ resp. $H_{s,t}$. Then, recalling the definition of $J_{T,H}$ and using that the weak derivative of $H_{s,t}$ is given by:
$$\int_0^T f(r)\partial_r H_{s,t}(r) dr = f(T)H_{s,t}(T)-f(0)H_{s,t}(0) - \log(1/(t-s))(f(t)-f(s)), $$
we have
\begin{equation}
    \begin{split}
        \underset{p\rightarrow \infty}{\overline{\lim}} J_{T,H_{s,t}^p}(\pi) &= \log(1/(t-s))\big(a(t)-a(s)\big)\\
        &- \int_s^t \left[2\gamma\big(1-a(r)\big)\left(\frac{1}{t-s}-1 \right) + 2(1-\gamma)a(r)(t-s-1) \right]dr.
    \end{split}
\end{equation}
Now, as $\lims J_{T,H_{s,t}^p}(a) \leq I_T(a)<\infty $ and $a$ is bounded, there are constants $C_1,C_2>0$ such that
\begin{equation}\label{borne}
    \left\vert a(t) -a(s)\right\vert \leq \big[\log(1/(t-s))\big]^{-1} \left[I_T(a) + (t-s)C_1 + C_2 \right]
\end{equation}
and the right hand side of this inequality goes to zero as $s\rightarrow t$, hence the result.
\end{proof}
For $q\geq 0$, the $q$-level set of $I$ is defined as
 $$ E_q=\left\{a\in \mathcal{D}_{[0,1]}^T~ I_T(a)\leq q \right\}.$$ 
 A corollary of the proof of Proposition \ref{Continu} is the following: 
\begin{lem} \label{Levelset} The level sets of $I_T$ are compact in $\mathcal{D}_{[0,1]}^T$.
\end{lem}
\begin{proof}
    Fix $q\geq 0$. By lower semi continuity of $I_T$, $E_q$ is closed so we just need to show that it is relatively compact. For that, we show that
$$\underset{\delta \rightarrow 0}{\lim}~ \underset{u\in E_q}{\sup}~\underset{\underset{|t-s|<\delta}{0\leq s,t\leq T}}{\sup}|a(t)-a(s)| = 0 $$
but this is an immediate consequence of \eqref{borne}.
\end{proof}

\begin{rem} \label{uniformconv} By Proposition \ref{Continu}, any level set of $I_T$ is a subset of $\mathcal{C}([0,T])$ and the proof of Lemma \ref{Levelset} shows that a level set is compact in $\mathcal{C}([0,T])$ for the topology of uniform convergence.
\end{rem}

\begin{prop}\label{Inversion}
    Fix $a\in \mathcal{C}^2([0,T])$ and consider $H\in \mathcal{C}^1([0,T])$ a solution of the following partial differential equation,
    \begin{equation}\label{EDP}
        \partial_t a(t) = 2\gamma\big(1-a(t)\big)e^{H(t)} - 2(1-\gamma)a(t)e^{-H(t)}.
    \end{equation}
    Then, $H$ solves the following variational problem
    $$I_T(a) = \underset{G\in \mathcal{C}^1([0,T])}{\sup}~ J_{T,G}(a) = J_{T,H}(a).  $$
\end{prop} 
\begin{proof}
    For $G\in \mathcal{C}^1([0,T])$, performing an integration by part we get
\begin{equation*}
        J_{T,G}(a) = \int_0^TG_s\partial_sa_sds ~ - \int_0^T A_G(a)(s)ds,
\end{equation*}
where recall that $A_G(a)$ was defined in \eqref{AG}. Now, writing $J_{T,G}(a) = J_{T,G}(a)-J_{T,H}(a) + J_{T,H}(a)$ we check that $J_{T,G}(a)-J_{T,H}(a)\leq 0 $. Replacing $\partial_s a_s$ thanks to \eqref{EDP} we have,
\begin{equation}
    \begin{split}
        J_{T,G}(a)-J_{T,H}(a) &= \int_0^T(G_s-H_s)\partial_s a_s ds - \int_0^T\big[A_G(a)(s)- A_H(a)(s)\big]ds\\
        &= \int_0^T (G_s-H_s)\big[2\gamma(1-a_s )e^{H_s}-2(1-\gamma)a_s e^{-H_s} \big]ds\\
        &- \int_0^T\Big[2\gamma(1-a_s)(e^{G_s}-e^{H_s}) + 2(1-\gamma)a_s(e^{-G_s}-e^{-H_s})] \Big]ds\\
        &=\int_0^TF_{a,H_s}(G_s)ds,
    \end{split}
\end{equation}
where for $\Check{a}\in [0,1]$, $x,y\in \R$,
\begin{equation*}
    F_{\Check{a},y}(x):= 2\gamma(1-\Check{a})\big[(x-y)e^y-(e^x-e^y) \big]-2\Check{a}(1-\gamma)\big[(x-y)e^{-y}+ e^{-x}-e^{-y} \big].
\end{equation*}
We have that $F_{\Check{a},y}$ is a concave function of $x$ and $F_{\Check{a},y}'(y)=0$ so it reaches its maximum at $y$, where $F_{\Check{a},y}(y)=0$. It follows that $\int_0^TF_{a,H_s}(G_s)ds \leq 0$, hence the result.
\end{proof}

We will often make use of the following result: for any $a\in \mathcal{D}_{[0,1]}^T$ such that $I_T(a)<\infty$, for any $t\in [0,T]$,
\begin{equation}\label{DecompoI}
    I_T(a) = I_t(a) + I_{T-t}(a(. + t)).
\end{equation}
In particular, as $I$ is a positive functional, for any $t\in [0,T]$ $I_t(a)\leq I_T(a)$. To prove that, decomposing $J_{T,G}(a)$, for $G\in \mathcal{C}^{\infty}([0,T])$, as 
\begin{equation*}
    \begin{split}
        J_{T,G}(a) = J_{t,G}(a) + J_{T-t,G}(a(.+t)),
    \end{split}
\end{equation*}
we immediately get that $I_T(a) \leq I_t(a) + I_{T-t}(a(. + t))$. To prove the reverse inequality, fix $\varepsilon>0$ and consider $G\in \mathcal{C}^1([0,t])$ and $H\in \mathcal{C}^1([0,T-t])$ such that
\begin{equation*}
    J_{t,G}(a) + J_{T-t,H}(a(.+t)) \geq I_t(a) + I_{T-t}(a(. + t))-2\varepsilon.
\end{equation*}
Taking smooth approximations $(K_n)_{n\geq 1}$ of $G \mathds{1}_{[0,t)} + H(.-t)\mathds{1}_{[t,T]} $ in $\mathcal{C}^1([0,T])$, we have that
$$I_T(a) \geq \underset{n\rightarrow \infty}{\overline{\lim}} J_{T,K_n}(a) =  J_{t,G}(a) + J_{T-t,H}(a(.+t)) \geq I_t(a) + I_{T-t}(a(. + t))-2\varepsilon. $$ This holds for every $\varepsilon$, thus the converse inequality.
\begin{cor}\label{Strongsol}
If $a\in \mathcal{D}_{[0,1]}^T$ then $I_T(a)=0$ if and only if $a$ is a weak solution of $\partial_ta = -2(a -\gamma)$. This implies that $a$ is a strong solution of that equation and is therefore in $\mathcal{C}^{\infty}([0,T])$.
\end{cor}
\begin{proof}
If $a\in \mathcal{D}_{[0,1]}^T$ is a solution of $\partial_ta = -2(a -\gamma)$, by Proposition \ref{Inversion}, $I_T(a) = J_H(a)$ with $H=0$, so $I_T(a)=0$. 

Conversely, fix $a \in \mathcal{D}_{[0,1]}^T $ and assume that $I_T(a)=0$. For $t\in [0,T]$, consider the function $F_t:\R \rightarrow \R$ defined by
$F_t(x)=J_{t,x}(a)$, where $x$ refers to the function that is constant in time, equal to $x$. Then
$$ F_t(x) = x(a_t-a_0) + 2\int_0^t\left[\gamma (e^x-1) - (1-\gamma)(e^{-x}-1) \right]a_s ds -2\gamma t (e^x-1). $$
By \eqref{DecompoI}, we have
$I_T(a) \geq I_t(a)$. As $I_T(a)=0$, for any $t\in [0,T]$, $I_t(a)=0$. Evaluating $J_G(a)$ with $G$ constant equal to $x$ yields that $F_t \leq 0$ on $\R$. As $F_t(0)=0$, $0$ is a local maximum for $F_t$ so $F_t'(0)=0$. This implies that for any $t\in [0,T]$,
$$a_t-a_0 =-2\int_0^ta_sds + 2\gamma, $$
so $a$ is a solution of $\partial_ta = -2(a -\gamma)$.
\end{proof}

\subsection{The $I_T$ -- density}

 The proof follows the same steps as in the seminal papers \cite{QRV} and \cite{BLM}. First we approximate trajectories by ones which follow the hydrodynamic equation on a small time interval (Lemma \ref{Pi1}) and are uniformly bounded away from $0$ and $1$ (Lemma \ref{PI2}). Then, we regularize such trajectories in time (Lemma \ref{PI3}). 
\begin{Def}
    A set $A\subset \mathcal{D}_{[0,1]}^T$ is said to be $I_T$ -- dense if for any $a$ in $\mathcal{D}_{[0,1]}^T$ such that $I_T(a)<\infty$, there exists a sequence $(a_p)_{p\geq 1}$ of elements in $A$ such that
    $$a_p \underset{p \rightarrow \infty}{\longrightarrow}a~ \text{in}~ ~ \mathcal{D}_{[0,1]}^T~ ~ ~ \text{and}~ ~ ~ I_T(a_p)\underset{p \rightarrow \infty}{\longrightarrow} ~ I_T(a). $$
     
\end{Def}
Let $B^T$ be the set of elements $a$  in $\mathcal{C}^2([0,T])$ for which there exists $H\in \mathcal{C}^1([0,T])$ such that $a$ and $H$ are related by the ordinary differential equation \eqref{EDP}.
\begin{thm} \label{I-dense}
The set $B^T$ is $I_T$ -- dense.
\end{thm}
For the proof, of Theorem \ref{I-dense}, we establish the $I_T$ -- density of some intermediate sets. 
  Denote by $\Pi_1^T$ the set of elements in $\mathcal{D}_{[0,1]}^T$ such that for any $a\in \Pi_1^T$, there is a $\delta>0$ such that on $[0,\delta]$, $a$ is a solution of the ordinary differential equation
\begin{equation}\label{hydordyn}
    \partial_t\lambda = -2\lambda + 2\gamma .
\end{equation}

\begin{lem}\label{Pi1}
    The set $\Pi_1^T$ is $I_T$ -- dense in $\mathcal{D}_{[0,1]}^T$.
\end{lem}

\begin{proof}

Consider $a$ in $\mathcal{D}_{[0,1]}^T$ such that $I_T(a)<\infty$ and 
$\lambda:[0,T] \rightarrow [0,1]$ the unique solution of \eqref{hydordyn} with initial condition $a(0)$. For $\delta>0$, define $a^{\delta}$ as follows:
\begin{equation} 
   a^{\delta}(t)= \left\{
    \begin{array}{ll}
        \lambda(t),~ ~\text{if}~ t\in [0,\delta]\\ 
        \lambda(2\delta-t),~ ~\text{if}~ t\in[\delta,2\delta]\\
        a(t-2\delta),~ ~\text{if}~ t\in [2\delta,T].
        \end{array}
\right.
\end{equation}
It is clear that $a^{\delta}$ converges to $a$ in $\mathcal{D}_{[0,1]}^T$ as $\delta \downarrow 0$ and, by construction, $a^{\delta}$ belongs to $\Pi_1^T$. By lower semi continuity of $v \mapsto I_T(v)$, 
$$I_T(a) \leq \underset{\delta \rightarrow 0}{\liminf} ~ I_T(a^{\delta}). $$
We are therefore left to show that
$$I_T(a) \geq \underset{\delta \rightarrow 0}{\limsup}~  I_T(a^{\delta}). $$
Decomposing  $J_{T,G}(a^{\delta})$ into the sum of the contributions on each time interval $[0,\delta]$, $[\delta,2\delta]$ and $[2\delta,T]$ we get:
\begin{equation}\label{Decompo}
    \begin{split}
       J_{T,G}(a^{\delta}) &= a^{\delta}_T G_T -a^{\delta}_{2\delta}G_{2\delta} -\int_{2\delta}^T \partial_s G_s a^{\delta}_s ds - \int_{2\delta}^TA_G(a^{\delta})(s)ds\\
        &+ a^{\delta}_{2\delta}G_{2\delta}- a^{\delta}_{\delta}G_{\delta}
        -\int_{\delta}^{2\delta} \partial_s G_s a^{\delta}_s ds - \int_{\delta}^{2\delta}A_G(a^{\delta})(s)ds\\
    &+a^{\delta}_{\delta}G_{\delta}- a^{\delta}_0 G_0 - \int_{0}^{\delta} \partial_s G_s a^{\delta}_s ds - \int_{0}^{\delta}A_G(a^{\delta})(s)ds,
    \end{split}
\end{equation}
where we recall that the definition of $A_G$ is given in \eqref{AG}. The first term is bounded above by $I_{T-2\delta}(a)$ and recall, by arguments detailed in the proof of Corollary \ref{Strongsol}, that $I_{T-2\delta}(a) \leq I_{T}(a) $. The last term in \eqref{Decompo} equals $I_{\delta}(a_{\delta})=I_{\delta}(\lambda)=0$ because $\lambda$ solves \eqref{hydordyn} on $[0,\delta]$. Finally, let us show that 
\begin{equation}
    \underset{\delta \rightarrow 0}{\limsup}~ \underset{G\in \mathcal{C}^1([0,T])}{\sup}~ \left\{a^{\delta}_{2\delta}G_{2\delta}- a^{\delta}_{\delta}G_{\delta}
        -\int_{\delta}^{2\delta} \partial_s G_s a^{\delta}_s ds - \int_{\delta}^{2\delta}A_G(a^{\delta})(s)ds \right\} = 0.
\end{equation}
For $G\in \mathcal{C}^1([0,T])$,
\begin{equation}
\begin{split}
        &a^{\delta}_{2\delta}G_{2\delta}- a^{\delta}_{\delta}G_{\delta}
        -\int_{\delta}^{2\delta} \partial_s G_s a^{\delta}_s ds -\int_{\delta}^{2\delta}A_G(a^{\delta})(s)ds = \int_{\delta}^{2\delta}  G_s \partial_s a^{\delta}_s ds - \int_{\delta}^{2\delta}A_G(a^{\delta})(s)ds\\
        &= \int_{0}^{\delta}  G_{2\delta-s} \partial_s a^{\delta}(2\delta-s) ds - \int_{0}^{\delta}A_G(a^{\delta})(2\delta-s)ds,
\end{split}
\end{equation}
therefore,
\begin{equation}
   \underset{G\in \mathcal{C}^1([0,T])}{\sup}~ \left\{a^{\delta}_{2\delta}G_{2\delta}- a^{\delta}_{\delta}G_{\delta}
        -\int_{\delta}^{2\delta} \partial_s G_s a^{\delta}_s ds -\int_{\delta}^{2\delta}A_G(a^{\delta})(s)ds \right\} =  \underset{G\in \mathcal{C}^1([0,T])}{\sup}~ \left\{\int_0^{\delta}G_s \partial_s\overline{a}_s ds -\int_0^{\delta}A_G(\overline{a})(s)ds \right\},
\end{equation}
where $\overline{a}(s) = a^{\delta}(\delta-s) = \lambda(\delta - s)$. Using that $a_0\in (0,1)$ and that $\lambda(t) = \gamma+ \left(a_0 - \gamma \right)e^{-2t}$ we have that for any $t>0$,
$$0< \gamma(1-e^{-2\delta})\leq \overline{a}(s) \leq \gamma + \left(1-\gamma\right)e^{-2\delta}<1.$$
We can therefore define the continuous function
\begin{equation}\label{H}
    H(t)= \log \left(\frac{\overline{a}'(t) + \sqrt{\overline{a}'(t)^2 + 16\gamma(1-\gamma)(1-\overline{a}(t))\overline{a}(t)} }{4\gamma(1-\overline{a}(t))} ~ \right)
\end{equation}
on $[0,\delta]$. Furthermore, this belongs to $\mathcal{C}^1([0,T])$, and $\overline{a}$ and $H$ are related by \eqref{EDP}. It follows by Proposition \ref{Inversion} applied to the time interval $[0,\delta]$ that
$$\underset{G\in \mathcal{C}^1([0,T])}{\sup} ~ \left\{\int_0^{\delta}G_s \partial_s\overline{a}_s ds -\int_0^{\delta}A_G(\overline{a})(s)ds \right\}=\int_0^{\delta}H_s \partial_s\overline{a}_s ds - \int_0^{\delta}A_H(\overline{a})(s)ds.$$
Using that $H_s \partial_s\overline{a}_s$ and $A_H(\overline{a})(s)$ are continuous functions, the second term converges to zero as $\delta\downarrow 0$, hence the result.
\end{proof}
Denote by $\Pi_2^T$ the set of elements $a$ in $\Pi_1^T$ such that for every $\delta>0$, $a$ is uniformly bounded away from $0$ and $1$ on $[\delta,T]$, that is, there is an $\varepsilon>0$, such that for all $t\in [\delta,T]$, $\varepsilon\leq a(t)\leq 1-\varepsilon$.
\begin{lem}\label{PI2}
    The set $\Pi_2^T$ is $I_T$ -- dense in $\Pi_1^T$.
\end{lem}
\begin{proof}
Fix an $a$ in $\Pi_1^T$ such that $I_T(a)<\infty$. For $\varepsilon>0$, introduce
$a^{\varepsilon}=(1-\varepsilon)a + \varepsilon \lambda$, 
where $\lambda:[0,T] \rightarrow [0,1]$ is the unique solution of \eqref{hydordyn} with initial condition $a(0)$. By definition, $a^{\varepsilon}$ belongs to $\Pi_1^T$ and $a^{\varepsilon}$
converges to $a$ in $\mathcal{D}_{[0,1]}^T$ as $\varepsilon \downarrow 0$. Furthermore, using that 
$$\lambda(t) = \gamma + \left(a_0-\gamma\right)e^{-2t},$$
with $a_0\in [0,1]$, the following inequalities hold: for any $\delta>0$ and $t\in [\delta,T]$,
$$0<\varepsilon \gamma\left(1-e^{-2\delta} \right)\leq a^{\varepsilon}(t) \leq 1-\varepsilon + \varepsilon\left[\gamma + \left(1-\gamma\right)e^{-2\delta} \right]<1  $$
so $a^{\varepsilon}$ is in $\Pi_2^T$. Again, by semi continuity of $I_T$, $\underset{\varepsilon\downarrow 0}{\liminf} I_T(a^{\varepsilon}) \geq I_T(a)$. Then, by linearity of $J_{T,G}$, $$J_{T,G}(a^{\varepsilon}) = (1-\varepsilon)J_{T,G}(a) + \varepsilon J_{T,G}(\lambda)\leq (1-\varepsilon)J_{T,G}(a)\leq (1-\varepsilon)I_T(a),$$
where we used that $J_{T,G}(\lambda)\leq 0$ in the first inequality. Taking the $\limsup$ we get the desired result.
\end{proof}

Denote by $\Pi_3^T$ the set of elements in $\Pi_2^T$ belonging to $\mathcal{C}^2([0,T])$. 
\begin{lem}\label{PI3}
    The set $\Pi_3^T$ is $I_T$ -- dense in $\Pi_2^T$.
\end{lem}
\begin{proof}
Consider $a\in \Pi_2^T$ such that $I_T(a)<\infty$. By Proposition \ref{Continu}, $a$ is continuous. Let $\delta>0$ be such that $a$ is solution to \eqref{hydordyn} on $[0,3\delta]$. Consider $\phi:\R \rightarrow \R$ smooth with compact support in $(0,1)$ and $\int_0^1\phi(s)ds=1$. For $\epsilon>0$, define
$\psi(\epsilon,s)= \frac{1}{\epsilon}\phi\left( \frac{s}{\epsilon}\right). $
Then, $\left(\psi(\epsilon,.) \right)_{\epsilon>0}$ is an approximation of the identity on compact sets in the sense that for any $\rho\in \mathcal{C}(\R)$ with compact support, $t\mapsto \int_{\R}\rho(t+s)\psi(\epsilon,s)ds$ converges uniformly to $\rho$ on $\R$ as $\epsilon\rightarrow 0$. Consider $\varepsilon:[0,T] \rightarrow [0,1]$ a smooth non decreasing function such that
\begin{equation} 
   \varepsilon(t)= \left\{
    \begin{array}{ll}
        0,~ ~\text{if}~ t\in [0,\delta]\\ 
        0<\varepsilon(t)<1,~ ~\text{if}~ t\in (\delta,2\delta)\\
        1,~ ~\text{if}~ t\in [2\delta,T]
        \end{array}
\right.
\end{equation}
and for $p\in \N$, define $\varepsilon_p(t)=\frac{\varepsilon(t)}{p}$. Introduce the sequence
$a^p(t) = \int_0^1a(t+\varepsilon_p(t)s)\phi(s)ds $, where we extend $a$ on $[T,T+1]$ by letting, for $t\in [0,1]$,
$a(t+T) = \Tilde{\lambda}(t) $,
where $\Tilde{\lambda}$ is the solution to \eqref{hydordyn} with initial condition $a(T)$. 

By construction of $\phi$, the sequence $a^p$ converges to $a$ in $\mathcal{D}_{[0,1]}^T$. For $t\in [0,\delta]$, $a^p(t)=a(t)$, where $a$ is solution to \eqref{hydordyn}, so $a^p\in \Pi_1^T$. As $a\in \Pi_2^T$, the convolution product ensures that $a^p$ is also in $\Pi_2^T$. For $t\in [0,2\delta)$, $a^p(t) =\int_0^t a(t+\varepsilon_p(t)s)\phi(s)ds$. As $a$ is smooth on $[0,3\delta)$, for $p$ large enough, $a^p$ is smooth on $[0,2\delta)$. For $t\in (\delta,T]$, $\varepsilon_p(t)>0$ and the following change of variable holds:
\begin{equation*}
    \begin{split}
        a^p(t) = \int_0^{\varepsilon_p(t)}a(t+s)\frac{1}{\varepsilon_p(t)}\phi\left(\frac{s}{\varepsilon_p(t)}\right)ds = \int_t^{t+\varepsilon_p(t)}a(s)\psi(\varepsilon_p(t),s-t)ds = \int_{\R}a(s)\psi(\varepsilon_p(t),s-t)ds
        \end{split}
\end{equation*}
and this is smooth in $t$. Hence, $a^p$ is smooth on $[0,T]$ and it follows that $a^p\in \Pi_3^T$. 

To conclude the proof, let us check that $\underset{p\rightarrow \infty}{\overline{\lim}} I_T(a^p) \leq I_T(a)$. For $G\in \mathcal{C}^1([0,T])$, decomposing $J_{T,G}(a^p)$ as in \eqref{Decompo} we have:
\begin{equation}\label{Decompoo}
    \begin{split}
       J_{T,G}(a^p) &= a^p_T G_T -a^{p}_{2\delta}G_{2\delta} -\int_{2\delta}^T \partial_s G_s a^{p}_s ds - \int_{2\delta}^TA_G(a^{p})(s)ds\\
        &+ a^{p}_{2\delta}G_{2\delta}- a^{p}_{\delta}G_{\delta}
        -\int_{\delta}^{2\delta} \partial_s G_s a^{p}_s ds - \int_{\delta}^{2\delta}A_G(a^{p})(s)ds\\
    &+a^{p}_{\delta}G_{\delta}- a^{p}_0 G_0 - \int_{0}^{\delta} \partial_s G_s a^{p}_s ds - \int_{0}^{\delta}A_G(a^{p})(s)ds.
    \end{split}
\end{equation}
Again, the last term is negative. The first term in \eqref{Decompoo} is given by 
$$\int_0^1\left[\tau_{\frac{s}{p}}a(T )G_T - \tau_{\frac{s}{p}}a (2\delta )G_{2\delta} - \int_{2\delta}^T\partial_s G_s \tau_{\frac{s}{p}}a (t)dt +\int_{2\delta}^T J_G(\tau_{\frac{s}{p}}a )(t)dt \right]\phi(s)ds $$
where $\tau_{\frac{s}{p}}a(t) = a\left(t + \frac{s}{p} \right)$. Bound this by $\int_0^1 I_T(\tau_{\frac{s}{p}}a)\phi(s)ds$. Using that $a$ solves \eqref{hydordyn} on $[T,T+1]$, this is less than $\int_0^1 I_T(a)\phi(s)ds = I_T(a)$. 

Now we deal with the second term in \eqref{Decompoo}. By regularity of $a^p$ we can perform the following integration by part
$$a^{p}_{2\delta}G_{2\delta}- a^{p}_{\delta}G_{\delta}
        -\int_{\delta}^{2\delta} \partial_t G_t a^{p}_t dt =  \int_{\delta}^{2\delta}\partial_t a_s^p G_s ds.   $$
Therefore,
\begin{equation*}
        \begin{split}
       &a^{p}_{2\delta}G_{2\delta}- a^{p}_{\delta}G_{\delta}
        -\int_{\delta}^{2\delta} \partial_t G_t a^{p}_t dt - \int_{\delta}^{2\delta}A_G(a^p)(s)ds= \int_{\delta}^{2\delta}\partial_s a_s^p G_s ds - \int_{\delta}^{2\delta}A_G(a^p)(s)ds.
    \end{split}
\end{equation*}
Using Proposition \ref{Inversion} and the fact that $a^p$ is smooth on $[\delta,2\delta]$ and uniformly bounded away from $0$ and $1$, we have that
\begin{equation}\label{HHH}
    \underset{\Tilde{G}\in \mathcal{C}^1([0,T])}{\sup}\int_{\delta}^{2\delta}\partial_s a_s^p \Tilde{G}_s ds - \int_{\delta}^{2\delta}A_{\Tilde{G}}(a^p)(s)ds = \int_{\delta}^{2\delta}\partial_s a_s^p H_s ds - \int_{\delta}^{2\delta}A_H(a^p)(s)ds
\end{equation}
where
$$H(t)=\log\left(\frac{(a^p)'(t) + \sqrt{(a^p)'(t)^2 + 16\gamma(1-\gamma)(1-a^p(t))a^p(t)} }{4\gamma(1-a^p(t))} ~ \right). $$
Since $a$ solves the ODE \eqref{EDPmasse} on $[0,3\delta]$,
\begin{equation*}
    \begin{split}
        \partial_t a_t^p &= \int_{\R} \partial_t a(t+s) \psi(\varepsilon_p(t),s) ds + \int_{\R}a(t+s)\partial_t \psi(\varepsilon_p(t),s) ds = -2 \int_{\R} a(t+s) \psi(\varepsilon_p(t),s) ds  + 2\gamma + r^p(t)
        \end{split}
\end{equation*}
with
$$r^p(t) := \int_{\R}a(t+s)\partial_t\psi(\varepsilon_p(t),s)ds. $$
Hence, 
\begin{equation}\label{r}
    \partial_t a_t^p = -2 a_t^p + 2\gamma +r^p(t),
\end{equation}
and the right hand side term in \eqref{HHH} equals
\begin{equation*}
    \begin{split}
       \int_{\delta}^{2\delta}\left(-2a^p(s) +2\gamma\right)H(s)ds + \int_{\delta}^{2\delta} r^p(s)H(s)ds - \int_{\delta}^{2\delta}A_H(a^p)(s)ds.
    \end{split}
\end{equation*}
Therefore, the second term in \eqref{Decompoo} is less than
 \begin{equation}
     \begin{split}
          \int_{\delta}^{2\delta}\left(-2a^p(s) +2\gamma\right)H(s)ds + \int_{\delta}^{2\delta} r^p(s)H(s)ds - \int_{\delta}^{2\delta}A_H(a^p)(s)ds. 
     \end{split}
 \end{equation}
Now, 
$$\int_{\delta}^{2\delta}\left(-2a^p(s) +2\gamma\right)H(s)ds  - \int_{\delta}^{2\delta}A_H(a^p)(s)ds\leq 0.$$
Indeed, for any $G\in \mathcal{C}^1([0,T]))$ and $a\in C([0,T])$ with values in $[0,1]$, for any $s\in [0,T]$,
\[ A_G(a)(s) +2(a-\gamma)G(s) \geq 0.  \]
To see that, for $s\in [0,T]$, write $A_G(a)(s) +2(a_s-\gamma)G_s$ as $f_{a_s}(G_s)$, where $f_a$ is a convex function achieving its minimum at $0$ where it is vanishes.

To conclude, we show that $\underset{p\rightarrow \infty}{\lim}\int_{\delta}^{2\delta}r^p(s)H(s)ds = 0$. Recall that $\int_0^1\phi(s)ds =1$, so $\int_{\R} \partial_t\psi(\varepsilon_p (t),s)ds = 0.$
We then have
$$r^p(t) = \int_{\R}\left[a(t+s)-a(t) \right]\partial_t\psi(\varepsilon_p(t),s)ds $$
 and as $a$ is smooth, so Lipschitz on $[\delta,2\delta]$, there is a $C(\delta)>0$ such that
$|a(t+s) - a(t)|\leq C(\delta)s$. Following the same lines as in the proof of Lemma 5.6 in \cite{Tsunoda-Landim}, we prove that $r^p(t) \leq \frac{C(\delta,\phi)}{p}$ which yields the desired result.
\end{proof}
\textit{Proof of Theorem \ref{I-dense}.} By Lemma \ref{PI3}, it is enough to check that $\Pi_3^T \subset B^T$. For that, fix $a\in \Pi_3^T$ and $t\in (0,T]$, and define
$$H(t) = \log\left(\frac{a'(t) + \sqrt{a'(t)^2 + 16\gamma(1-\gamma)(1-a(t))a(t)} }{4\gamma(1-a(t))} ~ \right). $$
This is well defined because $a$ is in $\Pi_2^T$. Since $a\in \mathcal{C}^2([0,T])$, $H$ is in $\mathcal{C}^1([0,T])$ and by construction of $H$, it is related to $a$ by \eqref{EDP}. Therefore $a$ belongs to $B^T$.

\subsection{Proof of Lemma \ref{Identific}}
 The proof of Lemma \ref{Identific} relies on an argument introduced in \cite{Bertini}. Let us recall some notation. The quasi potential $V:[0,1] \rightarrow [0,+ \infty]$ relatively to $I_T(~.~|\gamma)$ is defined by
\begin{equation}
    V(m) := \underset{T>0}{\inf} ~ ~ \underset{a(.),~ a(T) = m}{\inf}~ I_T\left(a|\gamma\right),
\end{equation}
where the second infimum is taken over elements $a$ of $\mathcal{C}^1([0,T])$. The aim is to show that $V=S$, where we recall that $S$ was defined in \eqref{EntropieS}. 

To prove that $V(m)\leq S(m)$, one exhibits a path $a$ connecting $\gamma$ to $m$ in $[0,T]$ such that $I_T(a|\gamma) \leq S(m)$. Recall the variational definition of $I_T$ given in Definition \ref{defGD}. To prove that $V(m)\geq S(m)$, we show that for any path $a$ connecting $\gamma$, at time $0$, to $m$ at time $T$, $I_T(a)\geq S(m)$ from which the inequality follows.

First we prove the following result:
\begin{lem} \label{ControlI1}
    For $\kappa \in (0,1)$, let $a_{\kappa}:[0,1] \rightarrow [0,1]$ be given by $a_{\kappa}(t) = \kappa t + \gamma$. Then,
    $$I_1(a_{\kappa}|\gamma) \underset{\kappa \rightarrow 0}{\longrightarrow} 0.$$
\end{lem}
\begin{proof}
    Choose $\kappa \in [0,1)$ so that $t\mapsto \kappa t + \gamma$ is an element of $\Pi_3^1$ (that is, uniformly bounded away from $0$ and $1$ and in $\mathcal{C}^2$). Then,
    $I_1(\kappa t + \gamma|\gamma) = J_{1,H_{\kappa}}(\kappa t + \gamma), $  with
    $$H_{\kappa}(t) = \log \left(\frac{\kappa + \sqrt{\kappa^2 + 16\gamma(1-\gamma)a(t)(1-a(t)) } }{4\gamma (1-a(t))} \right), $$
where $a(t) = \kappa t + \gamma$. Therefore,
\begin{equation*}
    \begin{split}
        I_1(\kappa t + \gamma|\gamma)
        &= \kappa\int_0^1H_{\kappa}(s)ds - \int_0^1\big[2\gamma(1-a(s))(e^{H_{\kappa}(s)}-1)+2(1-\gamma)a(s)(e^{-H_{\kappa}(s)}-1) \big]ds\\
        &= \kappa\int_0^1H_{\kappa}(s)ds - \int_0^1\big[2\gamma(1-\kappa s)(e^{H_{\kappa}(s)}-1)+2(1-\gamma)\kappa s(e^{-H_{\kappa}(s)}-1) \big]ds\\
        & - \int_0^1\big[2\gamma(1-\gamma)(e^{H_{\kappa}(s)}-1)+2(1-\gamma)\gamma(e^{-H_{\kappa}(s)}-1) \big]ds.
    \end{split}
\end{equation*}
The last term is less than $I_1(\gamma|\gamma)=0$. Using that $H_{\kappa}$ converges weakly to zero, that it is uniformly bounded in $\kappa$ and $t\in [0,t]$ and the dominated convergence theorem, the above converges to zero.  
\end{proof}
Let us now prove Lemma \ref{Identific}.
\begin{proof}
Let us first prove that $V(m)\leq S(m)$. Consider $m\in (0,1)$ and fix $0<\varepsilon<1$. Denote by $a^*$ the unique solution of the Cauchy problem: $\partial_t a^* =-2 a^* +  2\gamma $, $a^*(0)=m$. It is immediate to see that there is $T_1>0$ such that for any $t\geq T_1$, $|a^*(t)-\gamma|<\varepsilon$. Now, consider the following trajectory $m^*$ defined on $[0,T_1+1]$ by\begin{equation} 
   \left\{
    \begin{array}{ll}
        a^*(T_1)t + \gamma(1-t),~ \text{if}~ t\in [0,1]\\\\
        a^*(T_1+1-t),~ ~\text{if}~ t\in [1,T_1+1].
        \end{array}
\right.
\end{equation}
By definition, $V(m) \leq I_{T_1+1}(m^*|\gamma)$, and by \eqref{DecompoI},
\begin{equation}\label{shift}
\begin{split}
        I_{T_1+1}(m^*|\gamma) &= I_1(m^*|\gamma)+ I_{T_1}(m^*(.+1))\\ &=I_1(m^*|\gamma) + I_{T_1}(a^*(T_1-.)).
\end{split}
\end{equation}
Let us compute the second term in \eqref{shift} . Denote by $a(s)=a^*(T_1-s)$ for $s\in [0,T_1]$. Then $a$ satisfies the Cauchy problem $\partial_sa = 2a-2\gamma$, $a(0)=a^*(T_1)$, and one can check that $H$, as defined in \eqref{H} associated to $a$ is given by:
\begin{equation}
    H(t)= \log \left(\frac{\left(1-\gamma\right)a(t)}{\gamma\left(1-a(t)\right)} \right)=:\log \left(\frac{v(t)}{w(t)} \right).
\end{equation}
We claim that 
\begin{equation}\label{final}
    J_{T_1,H}(a^*(T_1-.))= m\log\left( \frac{(1-\gamma)m}{\gamma(1-m)}\right)- a^*(T_1)\log\left( \frac{(1-\gamma)a^*(T_1)}{\gamma(1-a^*(T_1))}\right) + \log\left( \frac{1-m}{1-\gamma}\right).
\end{equation}
The proof relies on a long but straightforward computation detailed in Appendix \ref{Appendice}.
Collecting \eqref{shift} and \eqref{final} we have
$$V(m) \leq I_1\big((a^*(T_1)-\gamma)t+ \gamma|\gamma\big) + m\log\left( \frac{(1-\gamma)m}{\gamma(1-m)}\right)- a^*(T_1)\log\left( \frac{(1-\gamma)a^*(T_1)}{\gamma(1-a^*(T_1))}\right) + \log\left( \frac{1-m}{1-\gamma}\right). $$
Take $\varepsilon \rightarrow 0$ and $T_1 \rightarrow \infty$ so that $a^*(T_1) \rightarrow \gamma$. By Lemma \ref{Identific}, the first term on the right hand side of \eqref{final} goes to zero. The rest converges to $S(m)$ so $V(m) \leq S(m). $

Now, we prove that $V(m)\geq S(m)$. It is enough to show that for any $T>0$, for any $a\in \mathcal{D}_{[0,1]}^T$ such that $a(0)=\gamma$ and $a(T) = m$, $I_T(a|\gamma) \geq S(m)$. Fix $T>0$, suppose that $a$ is in $\Pi_3^T$ and define, for $t\in [0,T]$, $H(t) = \log \left(\frac{(1-\gamma)a(t)}{\gamma(1-a(t))} \right)$. By definition, $I_T(a|\gamma) \geq J_{T,H}(a)$. Now, let us compute $J_{T,H}(a)$:
\begin{equation*}
    J_{T,H}(a)= mH_T - \gamma H_0 - \int_0^T a_s\partial_s H_sds - \int_0^T A_H(a)(s)ds.
\end{equation*}
We have $\partial_s H_s = \frac{\partial_s a_s}{a_s(1-a_s)}$, so 
$- \int_0^T a_s\partial_s H_sds = \left[ \log(1-a_s)\right]_0^T. $
Furthermore, 
\begin{equation*}
    \begin{split}
        A_H(a)(s) &= 2\gamma \left[\frac{(1-\gamma)a_s}{\gamma(1-a_s)} + \frac{\gamma(1-a_s)}{(1-\gamma)a_s} \right]a_s + 2(1-2\gamma)a_s -2\gamma\frac{(1-a_s)a_s}{(1-\gamma)a_s}a_s - 2\frac{(a_s-\gamma)}{1-a_s}\\
        &= 2(1-\gamma)\frac{a^2_s}{1-a_s}-2\gamma(1-a_s) + 2a_s - 4\gamma a_s + \frac{2\gamma}{1-a_s}-\frac{2a_s}{1-a_s}=0.
    \end{split}
\end{equation*}
We are left with
\begin{equation*}
    J_{T,H}(a|\gamma)= mH_T - \gamma H_0 +\left[ \log(1-a_s)\right]_0^T = S(m).    
\end{equation*}
To extend this fact for any trajectory $a\in \mathcal{D}_{[0,1]}^T$ such that $a(0)=\gamma$ and $a(T)=m$, use the $I_T$ -- density of $\Pi_3^T$  in $\mathcal{D}_{[0,1]}^T$. Indeed, 
if $a\in \mathcal{D}_{[0,1]}^T$, using the approximation $a^{\delta}$ from the proof of Lemma \ref{Pi1} of $a$ and $a^{\delta,\varepsilon}$ the approximation of each $a^{\delta}$ from the proof or Lemma \ref{PI2}, we have
$$\underset{\varepsilon, \delta \rightarrow 0}{\underline{\lim}}~ I_T(a^{\delta,\varepsilon}|\gamma) = I_T(a|\gamma) \geq S(m). $$
\end{proof}

\section{Dynamical large deviations principle}
In this section, we prove the dynamical large deviations principle (Theorem \ref{DynamicGD}) following the approach in \cite{QRV}, or \cite{BLM}. The steps are by now standard. For the upper bound, we use an exponential martingale and, for the lower bound, we perturb the dynamics to turn typical a trajectory.
\subsection{Large deviations upper bound}
We first prove the upper bound for compact sets.
For that, we use an exponential martingale as well as the superexponential replacement lemma (Lemma \ref{RL}). To extend the result to closed sets, we prove exponential tightness of the process (Proposition \ref{expotight}).
\subsubsection{Upper bound for closed compact sets}

For $G\in \mathcal{C}^1([0,T])$ consider, for $t\in [0,T]$, $F(t,\e_t) = N\langle\pi_t^N,G_t\rangle  = N\widehat{m}\left( \pi_t^N\right) G_t$ and
$$M_t(G) = \exp\left\{F(t,\e_t)-F(0,\e_0)-\int_0^te^{-F(s,\e_s)}(\partial_s+L)e^{F(s,\e_s)}ds \right\}. $$
Then $(M_t(G))_{0\leq t \leq T} $ is an exponential martingale of mean $1$ with respect to the natural filtration and a computation yields:
\begin{equation}\label{Grossemg}
    \begin{split}
        M_t(G) &= \exp\left\{N \left(\widehat{m}\left( \pi_t^N\right)G_t - \widehat{m}\left( \pi_0^N\right) G_0 - \int_0^t \partial_s G_s \widehat{m}\left( \pi_s^N\right) ds\right.\right.\\
        &\left.\left.-\int_0^t \sum_{x\in \{1,N-1\}}\left[r_x(1-\e_s(x))(e^{G_s}-1) + (1-r_x)\e_s(x) (e^{-G_s}-1)\right]ds \right)\right\}.
    \end{split}
\end{equation}
Fix $m_0\in (0,1)$ and consider a sequence of configurations $(\e^N)_{N\geq 1}$ such that $\widehat{m}(\pi^N(\e^N))_{N\geq 1}$ converges to $m_0$. For $\delta >0$, introduce the following event:
\begin{equation}\label{petit}
    \begin{split}
        &S_{\delta}^{G} = \Big\{(\e_s)_{0\leq s\leq T} \in \mathcal{D}_{\Omega_N}^T,~  \Big|\int_0^T \sum_{x\in \{1,N-1\}}\left[r_x(1-\e_s(x))(e^{G_s}-1) + (1-r_x)\e_s(x)(e^{-G_s}-1)\right]ds\\
        &- \int_0^T\big[2\gamma\big(1-\widehat{m}(\pi_s^N)\big) \big(e^{G_s}-1\big) + 2\widehat{m}(\pi_s^N)(1-\gamma)\big(e^{-G_s}-1\big)\big]ds  \Big|<\delta  \Big\}.
    \end{split}
\end{equation}
Let $\mathcal{F}$ be a closed subset of $ \mathcal{D}_{[0,1]}^T$. Introduce
$$\mathcal{H}_{m_0,\delta} = \Big\{a\in \mathcal{D}_{[0,1]}^T, |a(0)-m_0|<\delta \Big\}. $$
Using inequality \eqref{loginequ}, we have
\begin{equation*}
    \begin{split}
        &\lims \frac{1}{N}\log~ \Tilde{\mathbb{P}}_{\delta_{\e^N}}\left[\left(\widehat{m}\left( \pi_t^N\right)\right)_{0\leq t \leq T}\in \mathcal{F}\cap \mathcal{H}_{m_0,\delta}\right]\\
        &\leq \max \left(~ \lims~  \frac{1}{N}\log~ \Tilde{\mathbb{P}}_{\delta_{\e^N}}\left[\left(\widehat{m}\left( \pi_t^N\right)\in \mathcal{F}\cap \mathcal{H}_{m_0,\delta}\right)\cap \left(\pi_t^N\in S_{\delta}^{G} \right)\right]\right.,\left. \lims~  \frac{1}{N}\log~ \Tilde{\mathbb{P}}_{\delta_{\e^N}}\left[\pi_t^N\in \big(S_{\delta}^{G}\big)^c  \right]\right),
    \end{split}
\end{equation*} where, from now on, we forget the subscript $0\leq t \leq T$ in $\left(\widehat{m}\left( \pi_t^N\right)\right)_{0\leq t \leq T} $ and $(\pi_t^N)_{0\leq t \leq T}$ in order to lighten the notation. By Lemma \ref{RL}, the second limit is $-\infty$. Now, writing
$$\Tilde{\mathbb{P}}_{\delta_{\e^N}}\left[\left(\widehat{m}(\pi_t^N)\in \mathcal{F}\cap \mathcal{H}_{m_0,\delta}\right)\cap \left(\pi_t^N\in S_{\delta}^{G} \right)\right]= \Tilde{\mathbb{E}}_{\delta_{\e^N}}\left[\mathds{1}_{\left(\widehat{m}(\pi_t^N)\in \mathcal{F}\cap \mathcal{H}_{m_0,\delta}\right)\cap \left(\pi_t^N\in S_{\delta}^{G} \right)} M_T^G (M_T^G)^{-1} \right], $$
using the fact that $M_T^G$ is a martingale with mean $1$ and upper bounding $(M_T^G)^{-1}$ yields that this is less than
\begin{equation}
    \begin{split}
        &\underset{\pi\in S_{\delta}^{G}\cap m^{-1}(\mathcal{F})  }{\sup}~\exp\left\{-N \left(\widehat{m}(\pi_T^N)G_T - \widehat{m}(\pi_0^N)G_0 - \int_0^T \partial_s G_s \widehat{m}(\pi_s^N)ds \right.\right.\\
        &-\left.\left.\int_0^T \big[2\gamma\big(1-\widehat{m}(\pi_s^N)\big) \big(e^{G_s}-1\big) + 2\widehat{m}(\pi_s^N)(1-\gamma)\big(e^{-G_s}-1\big)\big]ds \right)\right\} \exp(N\delta)
    \end{split}
\end{equation}
where we used the definition of the event $S_{\delta}^{G}$. Therefore, for any $G\in \mathcal{C}^1([0,T])$,
\begin{equation}
    \begin{split}
        \lims \frac{1}{N}\log~ \Tilde{\mathbb{P}}_{\delta_{\e^N}}\left[\left(\widehat{m}(\pi_t^N)\right)_{0\leq t\leq T}\in \mathcal{F}\cap \mathcal{H}_{m_0,\delta} \right] &\leq - \underset{a \in \mathcal{F}\cap \mathcal{H}_{m_0,\delta}}{\inf} \Big\{a_TG_T - a_0G_0 - \int_0^T \partial_s G_s a_s ds \\
        &-\int_0^T \big[2\gamma\big(1-a_s\big) \big(e^{G_s}-1\big) + 2a_s (1-\gamma)\big(e^{-G_s}-1\big)\big]ds \Big\} + \delta\\
        &\leq  -\underset{a \in \mathcal{F}\cap \mathcal{H}_{m_0,\delta}}{\inf}J_G(a) + \delta.
    \end{split}
\end{equation} Now, as $\widehat{m}(\pi^N(\e^N))_{N\geq 1}$ converges to $m_0$,
$$\lims \frac{1}{N}\log~ \Tilde{\mathbb{P}}_{\delta_{\e^N}}\left[\left(\widehat{m}\left( \pi_t^N\right)\right)_{0\leq t \leq T}\in \mathcal{F}\cap \mathcal{H}_{m_0,\delta}^c\right]=-\infty,$$
so we are left with
\begin{equation*}
    \lims \frac{1}{N}\log~ \Tilde{\mathbb{P}}_{\delta_{\e^N}}\left[\left(\widehat{m}\left( \pi_t^N\right)\right)_{0\leq t \leq T}\in \mathcal{F} \right] \leq   -\underset{a \in \mathcal{F}\cap \mathcal{H}_{m_0,\delta}}{\inf}J_G(a) + \delta
\end{equation*}
and taking $\delta \rightarrow 0$,
\begin{equation*}
    \lims \frac{1}{N}\log~ \Tilde{\mathbb{P}}_{\delta_{\e^N}}\left[\left(\widehat{m}\left( \pi_t^N\right)\right)_{0\leq t \leq T}\in \mathcal{F} \right] \leq   -\underset{a \in \mathcal{F}_{m_0}}{\inf}J_G(a) ,
\end{equation*}
where $\mathcal{F}_{m_0}$ is the set of elements of $\mathcal{F}$ with initial data $m_0$.
Optimizing this inequality over $G$ yields
\begin{equation}
    \begin{split}
        \lims \frac{1}{N}\log~ \Tilde{\mathbb{P}}_{\delta_{\e^N}}\left[\left(\widehat{m}(\pi_t^N)\right)_{0\leq t\leq T}\in \mathcal{F} \right] &\leq - \underset{G \in \mathcal{C}^1([0,T])}{\sup}~ \underset{a \in \mathcal{F}_{m_0}}{\inf}J_{T,G}\left(a\right).
    \end{split}
\end{equation}
If $\mathcal{F}$ is a compact subset of $\mathcal{D}_{\mathcal{M}}^T$, so is $\mathcal{F}_{m_0}$ and one can exchange the supremum and infimum. Indeed, at any fixed $G\in \mathcal{C}^1([0,T])$, $a \mapsto J_G(a)$ is a linear and continuous function so one can apply Varadhan's argument (as done in \cite{KL}). In that case we get
\begin{equation}\label{GD}
    \begin{split}
        \lims \frac{1}{N}\log~ \Tilde{\mathbb{P}}_{\delta_{\e^N}}\left[\left(\widehat{m}(\pi_t^N)\right)_{0\leq t\leq T}\in \mathcal{F} \right]  &\leq - \underset{a \in \mathcal{F}_{m_0} }{\inf}~ I_T\left(a\right) = - \underset{a \in \mathcal{F} }{\inf}~ I_T\left(a|m_0\right).
    \end{split}
\end{equation}
\subsubsection{Upper bound for closed sets}
In the previous subsection, we have established the large deviations upper bound for compact sets. To extend the result to closed sets we use the standard method presented in \cite{KL}, Chapter 10, and based on the so called exponential tightness of the process, stated in the following Proposition. 
\begin{prop} \label{expotight}
    For every $\ell \in \N$, there exists a compact set $K_{\ell}\in \mathcal{D}_{[0,1]}^T$ such that 
    \begin{equation}\label{Compact}
        \lims \frac{1}{N}\log ~ \Tilde{\mathbb{P}}_{\delta_{\e^N}}\left[\left(\widehat{m}(\pi_t^N)\right)_{0\leq t\leq T}\in K_{\ell}^c \right]\leq -\ell.
    \end{equation}
\end{prop}
To build a sequence of compact sets satisfying Proposition \ref{expotight}, one proves the following estimate:
\begin{lem}\label{AZ}
For $\varepsilon,\delta>0$, introduce the set
$$\mathcal{C}_{\delta,\varepsilon} = \{ a\in \mathcal{D}_{[0,1]}^T,~ \underset{s\leq t \leq s+\delta}{\sup}  |a_t-a_s|\leq \varepsilon \}. $$
For any $\varepsilon >0$,
\begin{equation}
    \underset{\delta \downarrow 0}{\lim}~  \lims~  \frac{1}{N}\log~ \Tilde{\mathbb{P}}_{\delta_{\e^N}}\left[\widehat{m}(\pi_t^N)\notin \mathcal{C}_{\delta,\varepsilon} \right]=-\infty.
\end{equation}
\end{lem}
We refer to \cite{KL}, or \cite{FN} for precise details on how to recover the large deviations upper bound for closed sets from Proposition \ref{expotight}, how to derive Proposition \ref{expotight} from Lemma \ref{AZ} and for a proof of Lemma \ref{AZ}.

\subsection{Large deviations lower bound}
To prove the lower bound in Theorem \ref{DynamicGD}, we follow the usual strategy which consists in finding a perturbation of the process under which a trajectory satisfying some regularity assumptions becomes typical. The large deviation functional then appears as the entropy of the measure induced by the perturbed process, relatively to the one induced by the initial process. To extend this to any trajectory, proceeding as in \cite{BLM}, we use an $I_T$-density argument.

\subsubsection{Hydrodynamic limit of a perturbation of the process}
Given $G\in \mathcal{C}^1([0,T])$,
consider the following generator:
\begin{equation}
    (\mathcal{L}_{N,b}^Gf)(\e) = \sum_{x\in \{1,N-1\}}\left[ e^{G(t)}r_x(1-\e(x)) + e^{-G(t)}\e(x)(1-r_x)\right]\left(f(\e^x)-f(\e)\right).
\end{equation}
We will write $G_t$ instead of $G(t)$. Denote  by $\{\overline{\e}_t^N,~ t\in [0,T]\}$ the Markov process with generator $N\overline{\mathcal{L}}_N$, where
$$\overline{\mathcal{L}}_N= N^2\mathcal{L}_{N,0} + \mathcal{L}_{N,b}^G.$$
Given $\mu_N$ a measure on $\Omega_N$, denote by $\overline{\mathbb{P}}_{\mu_N}^G$ the probability measure on $\mathcal{D}_{\Omega_N}^T$ induced by $(\overline{\e}_t^N)_{t\geq 0}$ when $\overline{\e}_0 \sim \mu_N$, and $\overline{\mathbb{E}}_{\mu_N}^G$ its associated expectation. In particular, for $\e^N\in \Omega_N$, $\overline{\mathbb{P}}_{\delta_{\e^N}}^G$ is the measure induced starting from $\e^N$. Denote by $\overline{\pi}_t^N$ the empirical measure associated to $\overline{\e}_t^N$ and $\widehat{m}(\overline{\pi}_t^N)$ its total mass.

\begin{thm} \label{LGNmassperturbed}(Hydrodynamic limit for the total mass of the perturbed process).
Fix $\overline{m}_0\in [0,1]$ and consider a sequence of measures $(\mu_N)_{N\geq 1}$ on $\Omega_N$ associated to the mass $\overline{m}_0$. Then, for any $t\in [0,T]$ and any $\delta> 0$,
    $$\underset{N\rightarrow \infty}{\lim}~ \overline{\mathbb{P}}_{\mu_N}^G\left[~ \Big|\widehat{m}(\pi_t^N) -\overline{m}(t)\Big|> \delta \right] =0, $$
 where $\overline{m}:[0,T] \rightarrow [0,1]$ is the unique solution of 
 \begin{equation} \label{EDPmasseperturbed}
   \left\{
    \begin{array}{ll}
        \partial_t \overline{m}(\pi)(t) = 2\gamma\big(1-\overline{m}(\pi)(t)\big)e^{G_t} - 2(1-\gamma)\overline{m}(\pi)(t)e^{-G_t}\\ 
        \overline{m}(0)=\overline{m}_0.
        \end{array}
\right.
\end{equation}
\end{thm}
The proof of Theorem \ref{LGNmassperturbed} follows the same lines as that of Theorem \ref{LGNmass}. Indeed, start by writing the martingale associated to $\widehat{m}(\pi_t^N) $
\begin{equation}
\begin{split}
        \overline{M}_t^N &:= \widehat{m}(\pi_t^N)  - \widehat{m}(\pi_0^N)   -N^3\inte\mathcal{L}_{N,0}\widehat{m}(\pi_s^N) \mathrm{d}{s}- N\inte \mathcal{L}_{N,b}^G \widehat{m}(\pi_s^N)  \mathrm{d}{s}\\
        &= \widehat{m}(\pi_t^N)  - \widehat{m}(\pi_0^N)  - \inte \sum_{x\in \{1,N-1\}}\left[e^{G_s}r_x(1-\overline{\e}_s(x))-e^{-G_s}\overline{\e}_s(x)(1-r_x) \right]ds.
\end{split}
\end{equation}
Using the quadratic variation of $\overline{M}_t^N$ one proves tightness of the sequence of measures $\overline{\mathbb{P}}_{\mu_N}^G$. Then, one proves the replacement lemma: given $g\in \mathcal{C}([0,T])$, for any $t\in [0,T]$ and $\delta>0$, for any $x\in \{1,N-1\}$,
\begin{equation}\label{RLperturb}
    \underset{N\rightarrow \infty}{\overline{\lim}}~  \frac{1}{N} \log ~ \overline{\mathbb{P}}_{\mu_N}^G\left[~ \Big|\inte g(s)(\e_s(x)-\widehat{m}(\pi_s^N) )\mathrm{d}{s} \Big|>\delta \right] = -\infty.
\end{equation}
From this, one proves that a limit point must lie on a solution of \eqref{EDPmasseperturbed}. Uniqueness of the solution of \eqref{EDPmasseperturbed} is immediate.

For the proof of \eqref{RLperturb}, use that there is a uniform constant $C_0$ such that 
$$ \langle \mathcal{L}_{N,b}^{G,t} \sqrt{f},\sqrt{f} \rangle_{\nurhon} \leq C_0.$$
To prove that, proceed as in the proof of (ii) in Lemma \ref{Ineg} and use that the transition rates of the dynamics are bounded in $N$. Then, to recover \eqref{RLperturb}, proceed as in the proof of Lemma \ref{RL} using this estimate.

\subsubsection{Proof of the lower bound}
Fix $m_0$ in $[0,1]$ and consider a sequence $(\e^N)_{N\geq 1}$ such that $\widehat{m}(\pi^N(\e^N))_{N\geq 1}$ converges to $m_0$. We wish to prove that for any open subset $\mathcal{O}$ of $\mathcal{D}_{[0,1]}^T$,
\begin{equation*}
    \underset{N\rightarrow \infty}{\underline{\lim}}~  \frac{1}{N}\log \Tilde{\mathbb{P}}_{\delta_{\e^N}}\left[\widehat{m}(\pi^N)\in \mathcal{O} \right]\geq - \underset{a\in \mathcal{O}}{\inf}I\left(a|m_0\right).
\end{equation*}
Consider a profile $a\in \mathcal{O}\cap \Pi_3^T $ such that $I_T\left(a|m_0\right)<\infty$. For $\varepsilon>0$ denote by $\widehat{B}_{\varepsilon}(a)$ the ball of radius $\varepsilon$ and centered in $a$ for the Skorohod distance. Consider $H\in \mathcal{C}^1([0,T])$ such that $a$ and $H$ are related by \eqref{EDP}. As $\mathcal{O}$ is open, there is $\varepsilon>0$ such that $\widehat{B}_{\varepsilon}(a)\subset \mathcal{O}$, so
\begin{equation}
    \begin{split}
        \Tilde{\mathbb{P}}_{\delta_{\e^N}}\left[\widehat{m}(\pi^N)\in \mathcal{O}\right]&\geq \Tilde{\mathbb{P}}_{\delta_{\e^N}}\left[\widehat{m}(\pi^N)\in \widehat{B}_{\varepsilon}(a) \right] = \overline{\mathbb{E}}_{\delta_{\e^N}}^H\left[~ \frac{d\Tilde{\mathbb{P}}_{\delta_{\e^N}}}{d\overline{\mathbb{P}}_{\delta_{\e^N}}^H} \mathds{1}_{\widehat{m}(\pi^N)\in \widehat{B}_{\varepsilon}(a)} \right]\\
        & = \overline{\mathbb{E}}_{\delta_{\e^N}}^H\left[~ \left.\frac{d\Tilde{\mathbb{P}}_{\delta_{\e^N}}}{d\overline{\mathbb{P}}_{\delta_{\e^N}}^H}\right| \widehat{m}(\pi^N)\in \widehat{B}_{\varepsilon}(a)\right] \overline{\mathbb{P}}_{\delta_{\e^N}}^H \left[\widehat{m}(\pi^N)\in \widehat{B}_{\varepsilon}(a) \right]
    \end{split}
\end{equation}

The Radon-Nikodym derivative in the expectation is given by:
\begin{equation}
    \begin{split}
        \left( M_T(H) \right)^{-1}
        &=  \exp\left\{-N \left(\widehat{m}_T(\pi^N)H_T - \widehat{m}_0(\pi^N)H_0 - \int_0^T \partial_s H_s \widehat{m}_s(\pi^N)ds\right.\right.\\
        &-\left.\left.\int_0^T \sum_{x\in \{1,N-1\}}\left[r_x(1-\overline{\e}^N_{s}(x))(e^{H_s}-1) + (1-r_x)\overline{\e}^N_{s}(x) (e^{-H_s}-1)\right]ds \right)\right\}.
    \end{split}
\end{equation}
For a justification of that, see for instance \cite{KL}, Appendix 1. Therefore,
\begin{equation}
    \begin{split}
        \frac{1}{N}\log \Tilde{\mathbb{P}}_{\delta_{\e^N}}\left[\widehat{m}(\pi^N)\in \mathcal{O} \right] &\geq \frac{1}{N}\log~  \overline{\mathbb{E}}_{\delta_{\e^N}}^{H,\varepsilon} \left[ M_T(H) ^{-1}\Big| \widehat{m}(\pi^N)\in \widehat{B}_{\varepsilon}(a) \right] + \frac{1}{N} \log~ \overline{\mathbb{P}}_{\delta_{\e^N}}^{H}\left[\widehat{m}(\pi^N)\in \widehat{B}_{\varepsilon}(a)  \right].
    \end{split}
\end{equation}
By the hydrodynamic limit of the perturbed process (Theorem \ref{LGNmassperturbed}), the last term goes to $0$ when $N \rightarrow \infty$. Using Jensen's inequality in the conditional expectation, the replacement lemma \ref{RL} and the martingale computations in Subsection 5.1.1,  we get that the right hand side is bounded below by $-J_{T,H}\left(a\right)=-I_T\left(a\right)$. Therefore,
\[ \underset{N \rightarrow \infty}{\underline{\lim}}~ \frac{1}{N}\log \Tilde{\mathbb{P}}_{\delta_{\e^N}}\left[\widehat{m}(\pi^N)\in \mathcal{O} \right]\geq -I(a)~ ~ \forall a\in \mathcal{O}\cap \Pi_3^T. \]
Optimize this in $a\in \mathcal{O}\cap \Pi_3^T$ and the lower bound follows by $I_T$ - density of $\Pi_3^T$ (see Lemma \ref{PI3}). Again, we refer to \cite{KL}, Chapter 10 for more details.

\section{Static large deviations}

In this section, we prove Theorem \ref{StaticGD}, that is, that the quasi potential $V$ is the large deviations functional for the total mass under the stationary profile. First we prove the lower bound, then the upper bound. In both cases we make use of the dynamic large deviations principle (Theorem \ref{DynamicGD}). In the lower bound we use the hydrostatic limit for the total mass (Theorem \ref{LGNequ}) and for the upper bound, inspired by \cite{FW}, \cite{BodineauGiac} and \cite{Farfan}, we use a Markov chain representation of the invariant measure.
For $\delta>0$, denote by
\begin{equation}\label{boule}
  B_{\delta} = \left\{\pi \in \mathcal{M},~ |\widehat{m}(\pi)-\gamma|<\delta \right\}.   
\end{equation}

\subsection{Proof of the lower bound}
Fix an open set $\mathscr{O}$ in $[0,1]$. By definition of $V$, it is enough to prove that for any $m\in \mathscr{O}$, for any $T>0$ and for any $a\in \mathcal{C}([0,T])$ such that $a(0)=\gamma$ and $a(T)=m$,
$$\underset{N \rightarrow \infty}{\underline{\lim}}~ \frac{1}{N} \log \mu_{ss}^N\left[\widehat{m}(\pi^N)\in \mathscr{O} \right] \geq -I_T(a|\gamma). $$
Fix $\widehat{a}\in \mathcal{C}([0,T])$ such that $\widehat{a}(0)=\gamma$ and $\widehat{a}(T)=m$. Recall (see Remark \ref{R1}) that there is a sequence $(\varepsilon_N)_{N\geq 1} \downarrow 0$ such that $\mu_{ss}^N\left(B_{\varepsilon_N} \right)$ converges to $1$. Consider such a sequence $(\varepsilon_N)_{N\geq 1}$. By stationarity of $\mu_{ss}^N$,
\begin{equation*}
    \begin{split}
        \mu_{ss}^N\left[\widehat{m}(\pi^N)\in \mathscr{O} \right] &= \Tilde{\mathbb{E}}_{\mu_{ss}^N}\left[\Tilde{\mathbb{P}}_{\delta_\e}\left(\widehat{m}(\pi_T^N)\in \mathscr{O}\right) \right] \\
        &\geq\Tilde{\mathbb{E}}_{\mu_{ss}^N}\left[\Tilde{\mathbb{P}}_{\delta_\e}\left(\widehat{m}(\pi_T^N)\in \mathscr{O} \right) \mathds{1}_{\e\in (\pi^N)^{-1}\left(B_{\varepsilon_N}\right) } \right]\\
        & \geq \mu_{ss}^N\left(B_{\varepsilon_N} \right) \underset{\e\in (\pi^N)^{-1}\left(B_{\varepsilon_N}\right)}{\inf} \Tilde{\mathbb{P}}_{\delta_\e}\left[ \pi_T^N\in \mathscr{O}\right]\\
        &\geq \frac{1}{2}~  \underset{\e\in (\pi^N)^{-1}\left(B_{\varepsilon_N}\right)}{\inf} \Tilde{\mathbb{P}}_{\delta_\e}\left[ \pi_T^N\in \mathscr{O} \right],
    \end{split}
\end{equation*}
where we used that for $N$ large enough, $\mu_{ss}^N\left(B_{\varepsilon_N} \right) \geq \frac{1}{2}$. Now, since $(\pi^N)^{-1}\left(B_{\varepsilon_N}\right)$ is finite, the infimum above is achieved for a certain $\e^N\in (\pi^N)^{-1}\left(B_{\varepsilon_N}\right)$, so
\begin{equation*}
\begin{split}
        \underset{N \rightarrow \infty}{\underline{\lim}}~ \frac{1}{N} \log \mu_{ss}^N\left[\widehat{m}(\pi^N)\in \mathscr{O}\right] &\geq \underset{N \rightarrow \infty}{\underline{\lim}}~ \frac{1}{N} \log~  \Tilde{\mathbb{P}}_{\delta_{\e^N}}\left[ \widehat{m}(\pi_T^N) \in \mathscr{O} \right]\\
        &\geq \underset{N \rightarrow \infty}{\underline{\lim}}~ \frac{1}{N} \log~  \Tilde{\mathbb{P}}_{\delta_{\e^N}}\left[ \widehat{m}\left((\pi_{t}^N)_{t\in [0,T]} \right)\in \mathscr{O}_T \right]
\end{split}
\end{equation*}
where
$\mathscr{O}_T = \big\{ ~a \in \mathcal{D}_{[0,1]}^T, ~ a(T)\in \mathscr{O}\big\} $
is an open set because $\mathscr{O}$ is open. By the dynamic large deviations principle and since $\widehat{m}(\pi^N(\e^N))$ converges to $\gamma$,
$$\underset{N \rightarrow \infty}{\underline{\lim}}~ \frac{1}{N} \log~  \Tilde{\mathbb{P}}_{\delta_{\e^N}}\left[ \widehat{m}\left((\pi_{t}^N)_{t\in [0,T]} \right)\in \mathscr{O}_T \right] \geq - \underset{v\in \mathscr{O}_T}{\inf}~I_T(v|\gamma) \geq -I_T(\widehat{a}|\gamma). $$

\subsection{Proof of the upper bound}
Consider $\mathscr{F}$ a closed subset of $[0,1]$. If $\gamma\in \mathscr{F}$, $\underset{m\in \mathscr{F}}{\inf}V(m)=V(\gamma)=0$ and the upper bound follows. 

Now, let us deal with the case where $\gamma \notin \mathscr{F}$. There is a $\delta>0$ such that $[\gamma-3\delta, \gamma+ 3 \delta]\cap \mathscr{F}=\emptyset$. As mentioned above, the idea is to use a representation of the invariant measure $\mu_{ss}^N$ in terms of an invariant measure for an irreducible dynamics defined on a subset of $\Omega_N$, as done in \cite{FW}, \cite{BodineauGiac} and \cite{Farfan}. The subset considered is included in the set of configurations $\e$ such that 
\[|\widehat{m}(\pi^N(\e))- \gamma| <\delta,\]
that is, the configurations whose associated empirical measure is in $B_{\delta}$, defined in \eqref{boule}. 

Define the closed set
$$R_{\delta} = \{\pi \in \mathcal{M},~ 2\delta\leq|\widehat{m}(\pi) -\gamma|\leq 3\delta \}. $$
For any integer $N$ and any subset $A$ of $\mathcal{M}$, let $A^N = (\pi^N)^{-1}(A)\in \Omega_N$, and denote by $\tau_{A^N}: \mathcal{D}(\R^+,\R)\rightarrow \R^+$ the entry time in $A^N$ of $\e_t^N$, that is,
$$\tau_{A^N} = \inf\{t\geq 0,~ \e_t^N\in A^N \}. $$
We also denote by $\mathfrak{F}=\widehat{m}^{-1}(\mathscr{F})\in \mathcal{M}$, which is closed because $\widehat{m}$ is continuous and $\mathscr{F}$ is closed. Define $\partial B_{\delta}^N$ as the set of configurations $\e \in B_{\delta}^N$ such that there is a finite sequence $(\e^i)_{1\leq i \leq k}$ such that $\e^0\in R_{\delta}^N$, $\e^k=\e$ and
\begin{itemize}
    \item [(i)] $\e^{i}$ is obtained from $\e^{i-1}$ by a move which is allowed by the dynamics.
    \item [(ii)] for any $1\leq i <k$, $\e^{i}\notin B_{\delta}^N$.
\end{itemize}
Define
$$\tau_1= \inf \{t>0,~  \exists~ s<t,~ \e_s\in R_{\delta}^N~ \text{and}~ \e_t\in \partial B_{\delta}^N \}. $$

\begin{lem}
    The sequence $(\e_{\tau_k})_{k\geq 1}$, where $\tau_k$ is obtained by iterating $\tau_1$, is an irreducible Markov chain.
\end{lem}
\begin{proof}
   Consider $\e, \xi \in \partial B_{\delta}^N$ and $(\e^i)_{1\leq i \leq k}$ a path connecting $\e^0\in R_{\delta}^N$ to $\e$. By irreducibility of the original dynamics, there is a sequence $(\xi^{i})_{0\leq i \leq \ell}$ connecting $\xi^0=\xi$ to $\xi^{\ell}=\e^0$. Then, consider the sequence of configurations $z^0=\xi^0,...,z^{\ell}=\xi^{\ell}=\e^0, z^{\ell +1}=\e^1,.., z^{\ell+k} = \e^k=\e$ connecting $\xi$ to $\e$. From the path $z$ we can extract a sequence $\Tilde{z}^0=\xi,...,\Tilde{z}^p=\e $ in $\partial B_{\delta}^N$, such that
$$\Tilde{\mathbb{P}}_{\Tilde{z}^{i-1}}\left[\e_{\tau_1}=\Tilde{z}^{i} \right]>0. $$
Indeed, consider $j_0=0$ and for $i\geq 1$, let
$$j_{2i-1}=\underset{\underset{z^j\in R_{\delta}^N}{j>j_{2i-2}}}{\min}~ \{j\} ~ ~ ~~\text{and}~ ~ ~~j_{2i}=\underset{\underset{z^j\in \partial B_{\delta}^N}{j>j_{2i-1}}}{\min}~ \{j\}.  $$
The sequence $\Tilde{z}^{i}=z^{2j}$ satisfies the assumptions. 
\end{proof}

Since $\partial B_{\delta}^N$ is finite, the irreducible Markov chain thus defined has a unique invariant measure that we denote by $\nu_N$. Following \cite{FW}, \cite{BodineauGiac} and \cite{Farfan}, the stationary measure $\mu_{ss}^N$ can be written as follows: for every subset $A$ of $\Omega_N$,
\begin{equation}
    \mu_{ss}^N(A) = \frac{1}{C_N}\int_{\partial B_{\delta}^N}\Tilde{\mathbb{E}}_{\delta_\e}\left(\int_0^{\tau_1}\mathds{1}_{\e_s\in A}ds \right)d\nu_N(\e)
\end{equation}
with $C_N=\int_{\partial B_{\delta}^N}\Tilde{\mathbb{E}}_{\delta_\e}\left(\tau_1\right)d\nu_N(\e). $ 
Therefore,
\begin{equation}\label{MF}
    \begin{split}
        \mu_{ss}^N(\widehat{m}^{-1}(\mathscr{F})) &= \frac{1}{C_N}\int_{\partial B_{\delta}^N}\Tilde{\mathbb{E}}_{\delta_\e}\left(\int_0^{\tau_1}\mathds{1}_{\e_s\in (\pi^N)^{-1}\left(\widehat{m}^{-1}(\Tilde{\mathscr{F}})\right)}ds \right)d\nu_{N}(\e)\\
        & \leq \frac{1}{C_N} \underset{\e\in \partial B_{\delta}^N}{\sup}\Tilde{\mathbb{E}}_{\delta_\e}\left(\int_0^{\tau_1}\mathds{1}_{\e_s\in (\pi^N)^{-1}\left(\widehat{m}^{-1}(\Tilde{\mathscr{F}})\right)}ds \right)\\
        &\leq \frac{1}{C_N} \underset{\e\in \partial B_{\delta}^N}{\sup}\Tilde{\mathbb{P}}_{\delta_\e}\left[\tau_{\mathfrak{F}^N}<\tau_1 \right] \sup_{\e\in \mathfrak{F}^N}\Tilde{\mathbb{E}}_{\delta_\e}\left[\tau_1\right],
    \end{split}
\end{equation}
where the last inequality results from the strong Markov property.

For $N$ large, any trajectory in $\mathcal{D}(\R^+,\Omega_N)$ starting from $\mathfrak{F}^N$ has to perform at least one jump before reaching $\partial B_{\delta}^N$, because $[\gamma-3\delta, \gamma+ 3 \delta]\cap \mathscr{F}=\emptyset$. As the jump rates of the dynamics in the bulk are of order $N^3$ and those for the dynamics at the boundary are of order $N$, there is a constant $c>0$ depending on $\alpha$ and $\beta$ such that  $C_N>\frac{1}{cN^3}$. If the mesh size $1/N$ of $[0,1]$ is small enough, that is $N$ large enough, by continuity of $\widehat{m}$, any trajectory in $\mathcal{D}(\R^+,\Omega_N)$ starting in $\partial B_{\delta}^N$ satisfies $\tau_{R_{\delta}^N}\leq \tau_{\mathfrak{F}^N}$. That means that when $N$ is large enough, a flip in the configuration cannot lead directly from $\partial B_{\delta}^N$ to $\mathfrak{F}^N$ without passing through $R_{\delta}^N$. For that same reason, we also have that for $N$ large enough, any trajectory starting in $\mathfrak{F}^N$ satisfies $\tau_1 = \tau_{B_{\delta}^N}$ almost surely. The second supremum in \eqref{MF} is therefore bounded by $\sup_{\e\in \mathfrak{F}^N}\Tilde{\mathbb{E}}_{\delta_\e}\left[\tau_{B_{\delta}^N}\right]$ and using the strong Markov property, the first supremum satisfies, for $N$ large enough:
\begin{equation*}
    \begin{split}
        \underset{\e\in \partial B_{\delta}^N}{\sup}\Tilde{\mathbb{P}}_{\delta_\e}\left[\tau_{\mathfrak{F}^N}<\tau_1 \right] &= \underset{\e\in \partial B_{\delta}^N}{\sup}\Tilde{\mathbb{E}}_{\delta_\e}\left[\Tilde{\E}_{\e_{\tau_{R_{\delta}^N}}}\left[\mathds{1}_{\tau_{\mathfrak{F}^N}<\tau_1} \right]\mathds{1}_{\tau_{R_{\delta}^N}<\tau_{\mathfrak{F}^N}} \right]\\
        & \leq \sup_{\e\in R_{\delta}^N} \Tilde{\mathbb{P}}_{\delta_\e}\left[\tau_{\mathfrak{F}^N}<\tau_1 \right] = \sup_{\e\in R_{\delta}^N}\Tilde{\mathbb{P}}_{\delta_\e}\left[\tau_{\mathfrak{F}^N}<\tau_{B_{\delta}^N} \right].
    \end{split}
\end{equation*}
We are left with
$$\mu_{ss}^N(\widehat{m}^{-1}(\mathscr{F})) \leq cN^3 \sup_{\e\in R_{\delta}^N}\Tilde{\mathbb{P}}_{\delta_\e}\left[\tau_{\mathfrak{F}^N}<\tau_{B_{\delta}^N} \right] \sup_{\e\in \mathfrak{F}^N}\Tilde{\mathbb{E}}_{\delta_\e}\left[\tau_{B_{\delta}^N}\right]. $$
Now, to prove the upper bound, it is enough to prove the following result:

\begin{lem}\label{Lemprinc}
    \begin{itemize}
        \item [(i)] For every $\delta>0$,
        \begin{equation}
            \underset{N \rightarrow \infty}{\overline{\lim}}~ \frac{1}{N} \log \sup_{\e\in \mathfrak{F}^N}\Tilde{\mathbb{E}}_{\delta_\e}\left[\tau_{B_{\delta}^N}\right] \leq 0.
        \end{equation}
        \item [(ii)] For every $\varepsilon>0$, there exists $\delta>0$ such that
        \begin{equation}
            \underset{N \rightarrow \infty}{\overline{\lim}}~ \frac{1}{N}\log \sup_{\e\in R_{\delta}^N}\Tilde{\mathbb{P}}_{\delta_\e}\left[\tau_{\mathfrak{F}^N}<\tau_{B_{\delta}^N} \right] \leq -\inf_{m\in \mathscr{F}}V(m) + \varepsilon.
        \end{equation}
    \end{itemize}
\end{lem}
For that, use the following result proved later on.
\begin{lem} \label{Lemsecond}
    For every $\delta>0$, there is $T_0, C_0, N_0>0$ such that for any $N \geq N_0$ and for any integer $k>0$,
    \begin{equation}
        \sup_{\e\in \Omega_N} \Tilde{\mathbb{P}}_{\delta_\e} \left[\tau_{B_{\delta}^N} \geq k T_0 \right] \leq \exp\left(-kC_0N \right).
    \end{equation}
\end{lem}
\begin{proof}\textit{(Lemma \ref{Lemprinc})}
    \begin{itemize}
    \item [(i)]For $\delta>0$ and $T_0,C_0,N_0$ as in the statement of Lemma \ref{Lemsecond}, for every $N\geq N_0$ and every $\e\in \Omega_N$,
$$ \Tilde{\mathbb{E}}_{\delta_\e}\left[\tau_{B_{\delta}^N} \right] = T_0\Tilde{\mathbb{E}}_{\delta_\e}\left[\frac{\tau_{B_{\delta}^N}}{T_0} \right]\leq T_0 \sum_{k=0}^{\infty}~ \sup_{\e\in \Omega_N} \Tilde{\mathbb{P}}_{\delta_\e}\left[\tau_{B_{\delta}^N} \geq k T_0 \right] \leq  T_0 \sum_{k=0}^{\infty}\exp\left(-kC_0N \right) \leq \frac{T_0}{1-e^{-C_0N}}, $$
therefore,
$$\underset{N \rightarrow \infty}{\overline{\lim}}~ \frac{1}{N} \log \sup_{\e\in \mathfrak{F}^N}\Tilde{\mathbb{E}}_{\delta_\e}\left[\tau_{B_{\delta}^N}\right] \leq  \underset{N \rightarrow \infty}{\overline{\lim}}~ \frac{1}{N} \log ~ \frac{T_0}{1-e^{-C_0N}} = 0.  $$
\item[(ii)]Fix $\varepsilon>0$ and pick $\delta>0$, that will be specified according to $\varepsilon$ later on. Consider $T_0, C_0>0$ as in Lemma \ref{Lemsecond}. Since $\underset{m\in \mathscr{F}}{\inf}V(m)<\infty$, there is an integer $k>0$ such that $-kC_0<-\underset{m\in \mathscr{F}}{\inf}V(m)$ so taking $T_{\delta}= k T_0$ we get, by Lemma \ref{Lemsecond},
\begin{equation}\label{omega}
    \underset{N \rightarrow \infty}{\overline{\lim}}~ \frac{1}{N} \log \sup_{\e\in \Omega_N} \Tilde{\mathbb{P}}_{\delta_\e}\left[\tau_{B_{\delta}^N}\geq T_{\delta} \right]\leq -\underset{m\in \mathscr{F}}{\inf}V(m).
\end{equation}
Now write for $\e\in R_{\delta}^N$,
\begin{align*}
    \Tilde{\mathbb{P}}_{\delta_\e}\left[\tau_{\mathfrak{F}^N}<\tau_{B_{\delta}^N} \right]&= \Tilde{\mathbb{P}}_{\delta_\e}\left[\left(\tau_{\mathfrak{F}^N}<\tau_{B_{\delta}^N}\right)\cap \left( \tau_{B_{\delta}^N}< T_{\delta}\right)  \right] + \Tilde{\mathbb{P}}_{\delta_\e}\left[\left(\tau_{\mathfrak{F}^N}<\tau_{B_{\delta}^N}\right)\cap \left( \tau_{B_{\delta}^N}\geq  T_{\delta}\right)  \right]\\
    &\leq \Tilde{\mathbb{P}}_{\delta_\e}\left[\tau_{\mathfrak{F}^N}\leq T_{\delta} \right] + \Tilde{\mathbb{P}}_{\delta_\e}\left[ \tau_{B_{\delta}^N}\geq  T_{\delta}  \right].
\end{align*}
Using that
$\underset{N \rightarrow \infty}{\overline{\lim}}~ \frac{1}{N}\log(a_N+b_N) \leq \max\left(\underset{N \rightarrow \infty}{\overline{\lim}}~ \frac{1}{N}\log a_N, \underset{N \rightarrow \infty}{\overline{\lim}}~ \frac{1}{N}\log b_N \right) $
and \eqref{omega}, we are left with
$$\underset{N \rightarrow \infty}{\overline{\lim}}~ \frac{1}{N}\log \sup_{\e\in R_{\delta}^N}\Tilde{\mathbb{P}}_{\delta_\e}\left[\tau_{\mathfrak{F}^N}<\tau_{B_{\delta}^N} \right] \leq \max\left(-\underset{m\in \mathscr{F}}{\inf}V(m),\underset{N \rightarrow \infty}{\overline{\lim}}~ \frac{1}{N}\log \sup_{\e\in R_{\delta}^N}\Tilde{\mathbb{P}}_{\delta_\e}\left[\tau_{\mathfrak{F}^N}<T_{\delta} \right] \right) $$
and so it is enough to show that
$$\underset{N \rightarrow \infty}{\overline{\lim}}~ \frac{1}{N}\log \sup_{\e\in R_{\delta}^N}\Tilde{\mathbb{P}}_{\delta_\e}\left[\tau_{\mathfrak{F}^N}<T_{\delta} \right] \leq -\underset{m\in \mathscr{F}}{\inf}V(m)+ \varepsilon.$$
As $R_{\delta}^N $ is finite, there is a configuration $\e^N\in R_{\delta}^N $ such that
$$\sup_{\e\in R_{\delta}^N}\Tilde{\mathbb{P}}_{\delta_\e}\left[\tau_{\mathfrak{F}^N}<T_{\delta} \right]=\Tilde{\mathbb{P}}_{\delta_{\e^N}}\left[\tau_{\mathfrak{F}^N}<T_{\delta} \right]\leq ~ \Tilde{\mathbb{P}}_{\delta_{\e^N}}\left[\widehat{m}(\pi^N)\in \mathscr{F}_{\delta} \right],  $$
where $\mathscr{F}_{\delta}$ is the closed set of elements in $\mathcal{D}_{[0,1]}^{T_{\delta}}$ such that for $a$ in that set, there is $t\in [0,T_{\delta}]$ such that $a(t)$ or $a(t-)$ is in $\mathscr{F}$. By compactness of $\mathcal{M}$, $\pi^N(\e^N)$ contains a subsequence which converges to an element $\rho_{\delta}$ in $\mathcal{M}_0$. The continuity of the map $\widehat{m}$ ensures that $\widehat{m}(\rho_{\delta})\in R_{\delta}$. Along that converging subsequence, using the dynamical large deviation principle for the total mass, one has:
$$\underset{N \rightarrow \infty}{\overline{\lim}}~\frac{1}{N} \log \Tilde{\mathbb{P}}_{\delta_{\e^N}}\left[\widehat{m}(\pi^N)\in \mathscr{F}_{\delta} \right]\leq - \inf_{a\in \mathscr{F}_{\delta}}I_{T_{\delta}}\left(a|\widehat{m}(\rho_{\delta}) \right).  $$
Finally, we show that there is a $\delta>0$, such that
$$\inf_{a\in \mathscr{F}_{\delta}}I_{T_{\delta}}\left(a|\widehat{m}(\rho_{\delta}) \right) \geq \inf_{m\in \mathscr{F}}V(m)-\varepsilon. $$
Assume that this is not true so that for any $k>0$, there is $a^k\in \mathscr{F}_{\frac{1}{k}}\subset \mathcal{D}_{[0,1]}^{T_{\frac{1}{k}}}$ such that
$$I_{T_{\frac{1}{k}}}\left(a^k\big|\widehat{m}(\rho_{\frac{1}{k}}) \right) < \inf_{m\in \mathscr{F}}V(m)-\varepsilon. $$
Then, for any $k>0$, there is $0<T'_k\leq T_{\frac{1}{k}}$ such that $a^k(T'_k)\in \mathscr{F}$ or $a^k(T'_k-)\in \mathscr{F}$. Without loss of generality, we assume that $a^k(T'_k)\in \mathscr{F}$. We have
\begin{equation}\label{lhs}
    \inf_{m\in \Tilde{\mathscr{F}}}V(m) \leq V\left(a^k\left(T'_k\right)\right) = \inf_{\underset{u(0)=\gamma,~ a(T)=a^k(T'_k)}{a, T>0}}I_T\left(a|\gamma\right) \leq ~ I_{T'_k+1}\left(\Tilde{a}^k|\gamma\right) 
\end{equation}
where $\Tilde{a}^k$ is defined on $ [0,T'_k+1]$ by
\begin{equation} 
   \Tilde{a}^k= \left\{
    \begin{array}{ll}
        (1-t)\gamma + t\widehat{m}(\rho_{\frac{1}{k}}),~ \text{if}~ t\in [0,1]\\\\
        a^k(t-1),~ ~\text{if}~ t\in [1,T'_k+1].
        \end{array}
\right.
\end{equation}
Then, by \eqref{DecompoI},
\begin{equation}
    \begin{split}
        I_{T'_k+1}\left(\Tilde{a}^k|\gamma\right)&= I_1\left((1-t)\gamma + t\widehat{m}(\rho_{1/N})|\gamma\right)+  I_{T'_k}\left(a^k|\widehat{m}(\rho_{\frac{1}{k}})\right)\\
        &\leq I_1\left(\left(\widehat{m}(\rho_{\frac{1}{k}})-\gamma \right)t|\gamma\right)+ \inf_{m\in \mathscr{F}}V(m)-\varepsilon,
    \end{split}
\end{equation}
where we used that $T'_k\leq T_{\frac{1}{k}}$, so $I_{T'_k}\left(a^k|\widehat{m}(\rho_{\frac{1}{k}})\right)\leq I_{T_{\frac{1}{k}}}\left(a^k|\widehat{m}(\rho_{\frac{1}{k}})\right).$ 

By Lemma \ref{ControlI1}, for $k>0$ large enough, $I_1\left(\left(\widehat{m}(\rho_{\frac{1}{k}})-\gamma \right)t|\gamma\right)< \varepsilon$, so we are left with
$$\inf_{m\in \mathscr{F}}V(m) \leq V\left(a^k\left(T'_k\right)\right)< ~ \varepsilon + \inf_{m\in \mathscr{F}}V(m)-\varepsilon= \inf_{m\in \mathscr{F}}V(m) $$
which is a contradiction.

\end{itemize}
\end{proof}

To prove Lemma \ref{Lemsecond}, we use the following

\begin{lem}\label{interm}
    For every $\delta>0$, there is a $T>0$ such that
    $$\inf\left\{I_T(a),~ a\in \mathcal{D}_{[0,1]}^T~ \text{and}~ a(T)\notin ]\gamma-\delta,\gamma+ \delta[  \right\}>0. $$
\end{lem}

\begin{proof}
Suppose that the result is not true and consider $\delta>0$ such that for any $T>0$
  $\underset{a\in D_{\delta,T}}{\inf} I_T(a)=0,$ where $$D_{\delta,T}= \left\{a\in \mathcal{D}_{[0,1]}^T~ \text{and}~ a(T)\notin ]\gamma-\delta,\gamma+ \delta[ \right\}.   $$ 
  For $T>0$ consider $(a^{T,k})_{k\geq 0}$ a sequence in $D_{\delta,T}$ such that $I(a^{T,k}) \rightarrow 0$. Recall that by Lemma \ref{Levelset} and Remark \ref{uniformconv}, the level sets are compact for the uniform convergence topology so there is $a^T\in \mathcal{C}([0,T]) $ a limit of a strong converging subsequence of $(a^{T,k})_{k\geq 0}$ and, by lower semi-continuity of $I_T$, $I_T(a^T)=0$. By Corollary \ref{Strongsol}, $a^T$ is therefore a strong solution of $\partial_t a = -2(a -\gamma)$ and $a^T(T) \underset{T\rightarrow \infty}{\rightarrow} \gamma$. The strong convergence of $a^{T,k}$ to $a$ also implies that $|a^T(T)-\gamma| \geq \delta$ which yields a contradiction when $T$ is large enough.
\end{proof}

\begin{proof}\textit{(Lemma \ref{Lemsecond})}
    Fix a $\delta>0$. To prove the result, it is enough to show that there is $T_0,C_0$ and an $N_0>0$ such that for every $N>N_0$,
$$\underset{\e\in \Omega_N}{\sup}~ \Tilde{\mathbb{P}}_{\delta_\e}\left[ \tau_{B_{\delta}^N}\geq T_0 \right]\leq \exp(-C_0N).$$
Indeed, using induction and the strong Markov property, one then gets the desired result. Let us show that there is $T_0,C_0>0$ such that 
$$\underset{N \rightarrow \infty}{\overline{\lim}}~ \frac{1}{N}\log \underset{\e\in \Omega_N}{\sup}~ \Tilde{\mathbb{P}}_{\delta_\e}\left[ \tau_{B_{\delta}^N}\geq T_0 \right]\leq  -C_0.$$ Consider $\e^N\in \Omega_N$ such that
$$\underset{\e\in \Omega_N}{\sup}~ \Tilde{\mathbb{P}}_{\delta_\e}\left[ \tau_{B_{\delta}^N}\geq T_0 \right]= \Tilde{\mathbb{P}}_{\delta_{\e^N}}\left[ \tau_{B_{\delta}^N}\geq T_0 \right]. $$
We have, for $\Tilde{T}_0<T_0$,
\begin{equation*}
    \begin{split}
        \Tilde{\mathbb{P}}_{\delta_{\e^N}}\left[ \tau_{B_{\delta}^N}\geq T_0 \right] &\leq \Tilde{\mathbb{P}}_{\delta_{\e^N}}\left[ \tau_{B_{\delta}^N}> \Tilde{T}_0 \right]\\
        &\leq \Tilde{\mathbb{P}}_{\delta_{\e^N}}\left[ m(\pi^N)\in D_{\delta,\Tilde{T}_0}\right].
    \end{split}
\end{equation*}
By compactness of $\mathcal{M}$ and the fact that each configuration in $\Omega_N$ has at most one particle per site, $\pi^N(\e^N)$ contains a subsequence converging to some $\rho$ in $\mathcal{M}_0$. By continuity of $\widehat{m}$ and the fact that $D_{\delta,\Tilde{T}_0}$ is closed, using the dynamical large deviations principle we get, up to some extraction,
\begin{equation}
\begin{split}
        \underset{N \rightarrow \infty}{\overline{\lim}}~ \frac{1}{N}\log\Tilde{\mathbb{P}}_{\delta_{\e^N}}\left[ \tau_{B_{\delta}^N}\geq T_0 \right] &= \underset{N \rightarrow \infty}{\overline{\lim}}~ \frac{1}{N}\log\Tilde{\mathbb{P}}_{\delta_{\e^N}}\left[\widehat{m}(\pi^N)\in  D_{\delta,\Tilde{T}_0}\right]\\
        &\leq - \underset{a\in \mathcal{D}_{\gamma} }{\inf}~ I_{\Tilde{T}_0}(a|\rho).
\end{split}
\end{equation}
We then conclude thanks to Lemma \ref{interm}, by taking $\Tilde{T}_0$ large enough.
\end{proof}

\appendix
\section{Proof of \eqref{final} } \label{Appendice}
Let us compute $J_{T_1,H}(a^*(T_1-.))$:
\begin{equation*}
    \begin{split}
        J_{T_1,H}(a^*(T_1-.))= mH(T_1) - a^*(T_1)H(0) - \int_0^{T_1}a_s\partial_s H(s)ds - \int_0^{T_1}A_H(a)(s)ds.
    \end{split}
\end{equation*}
On the one hand,
\begin{equation*}
    \int_0^{T_1}a_s\partial_s H(s)ds = \int_0^{T_1}\frac{2(a(s)-\gamma)}{1-a(s)}ds = \int_0^{T_1}\frac{\partial_s a_s}{1-a_s}ds = -\left[\log\left(1-a(s)\right) \right]_0^{T_1}.
\end{equation*}
On the other hand,
\begin{equation*}
    A_H(a)(s) = 2\gamma\left[\frac{v(s)}{w(s)} + \frac{w(s)}{v(s)} \right]a(s) + 2(1-2\gamma)a(s) - 2\frac{w(s)}{v(s)}a(s) - 2\gamma \left(\frac{v(s)-w(s)}{w(s)}\right),
\end{equation*}
and 
$$2\gamma\int_0^{T_1}\frac{v(s)-w(s)}{w(s)}ds = 2\gamma \int_0^{T_1}\frac{2a(s)-2\gamma}{2\gamma(1- a(s))}ds =-\left[\log\left(1-a(s)\right) \right]_0^{T_1}. $$
We are left with\begin{equation}\label{sum}
    \begin{split}
        J_{H,T_1}(a^*(T_1-.))&= mH(T_1) - u^*(T_1)H(0) + 2\left[\log\left(1-u(s)\right) \right]_0^{T_1}\\
        &+ 2\gamma\int_0^{T_1} \frac{a(s)}{b(s)}u(s)ds + \left[2\gamma-2 \right]\int_0^{T_1}\frac{b(s)}{a(s)}u(s)ds + 2(1-2\gamma)\int_0^{T_1}u(s)ds.
    \end{split}
\end{equation}
As $\frac{a(s)}{v(s)}= \frac{1}{2-2\gamma}$,
$$2\left[\gamma-1 \right]\int_0^{T_1}\frac{w(s)}{v(s)}a(s)ds = 2\gamma \int_0^{T_1}a(s)ds - 2\gamma T_1, $$
so
\begin{equation}\label{2+3}
    2\left[\gamma-1 \right]\int_0^{T_1}\frac{w(s)}{v(s)}a(s)ds + 2(1-\gamma)\int_0^{T_1}a(s)ds=2\left[1-\gamma\right] \int_0^{T_1}a(s)ds - 2\gamma T_1.
\end{equation}
We also have:
$\frac{v(s)a(s)}{w(s)} = 2\left[1-\gamma\right]\frac{a^2(s)}{1-a(s)} = 2\frac{a(s) - \gamma}{1-a(s)}- 2\left[1-\gamma\right]a(s) + 2\gamma.$
Therefore,
\begin{equation}\label{1}
    2\gamma\int_0^{T_1} \frac{v(s)}{w(s)}a(s)ds = -\left[\log\left(1-a(s)\right) \right]_0^{T_1}-2\left[1-\gamma \right]\int_0^{T_1}a(s)ds + 2\gamma T_1.
\end{equation}
Collecting \eqref{sum}, \eqref{2+3} and \eqref{1} we are left with
\begin{equation}
    J_{T_1,H}(a^*(T_1-.))= m\log\left( \frac{(1-\gamma)m}{\gamma(1-m)}\right)- a^*(T_1)\log\left( \frac{(1-\gamma)a^*(T_1)}{\gamma(1-a^*(T_1))}\right) + \log\left( \frac{1-m}{1-\gamma}\right).
\end{equation}

\bibliographystyle{plain}
\bibliography{biblio.bib}

\end{document}